\documentclass{article}

\usepackage{authblk}
\usepackage[utf8]{inputenc}
\usepackage{amssymb, amsthm, amsmath, tcolorbox}
\usepackage{enumitem}
\usepackage{comment}

\PassOptionsToPackage{hyphens}{url}\usepackage[backref,colorlinks,citecolor=blue,bookmarks=true]{hyperref}

\theoremstyle{definition}
\newtheorem*{theorem*}{Theorem}
\newtheorem*{definition*}{Definition}

\newtheorem{theorem}{Theorem}[section]
\newtheorem{lemma}[theorem]{Lemma}
\newtheorem{proposition}[theorem]{Proposition}
\newtheorem{claim}[theorem]{Claim}
\newtheorem{remark}[theorem]{Remark}

\newtheorem{example}[theorem]{Example}

\newtheorem{corollary}[theorem]{Corollary}
\newtheorem{definition}[theorem]{Definition}

\newtheorem{problem}{Problem}

\newcommand{\F}{\mathbb{F}}

\newcommand{\C}{\mathbb{C}}
\newcommand{\sub}{\subseteq}
\newcommand{\sm}{\setminus}
\newcommand{\setmid}{\,\vert\,}

\renewcommand{\a}{\alpha}

\newcommand{\g}{\gamma}
\newcommand{\e}{\epsilon}
\renewcommand{\d}{\delta}

\DeclareMathOperator{\Span}{Sp}
\DeclareMathOperator{\Image}{Im}
\DeclareMathOperator*{\Exp}{\mathbb{E}}

\DeclareMathOperator{\rk}{rk}

\DeclareMathOperator{\bias}{bias}
\DeclareMathOperator{\udeg}{udeg}
\DeclareMathOperator{\ch}{char}
\DeclareMathOperator{\Poly}{Poly}
\DeclareMathOperator{\Forms}{Form}

\DeclareMathOperator{\Z}{Z}

\DeclareMathOperator{\Supp}{Supp}
\DeclareMathOperator{\rank}{rank}
\DeclareMathOperator{\codim}{codim}
\DeclareMathOperator{\Tr}{Tr}

\newcommand{\X}{\vec{X}}
\newcommand{\Y}{\vec{Y}}

\newcommand{\B}{\vec{B}}

\newcommand{\y}{\vec{y}}
\newcommand{\z}{\vec{z}}
\renewcommand{\P}{\vec{P}}
\newcommand{\U}{\vec{U}}

\let\vec\mathbf

\title{A weak regularity lemma for polynomials}
\author[1]{Guy Moshkovitz\footnote{G.~M. was partially supported by NSF Award DMS-2302988, NSF CAREER award DMS-2544490, and a PSC-CUNY award.  
D.~W. was supported as a graduate student in this research by the first grant.}}
\author[2]{Dora Woodruff}
\affil[1]{City University of New York}
\affil[2]{Massachusetts Institute of Technology}

\date{}

\begin{document}

\maketitle
\begin{abstract}
    A regularity lemma for polynomials provides a decomposition in terms of a bounded number of approximately independent polynomials. Such regularity lemmas play an important role in numerous results, yet suffer from the familiar shortcoming of having tower-type bounds or worse. In this paper we design a new, weaker regularity lemma with strong bounds. The new regularity lemma in particular provides tools for quantitatively studying the curves contained in the image of a polynomial map, which is beyond the reach of standard rank methods.

    The weak regularity lemma turns out to be powerful enough to yield results on arithmetic circuits and polynomial ranks that may be of independent interest:
    \begin{itemize}
        \item A general upper bound on the arithmetic circuit size of low-degree polynomial maps based solely on their image: if the image avoids curves of degree below $u$ then there is an arithmetic circuit of size $n^{\lfloor d/u \rfloor + o(1)}$, a power-saving bound compared to the typical $n^{d-o(1)}$ bound for degree-$d$ polynomials.
        \item An upper bound on the top fan-in of depth-4 arithmetic formulas under similar conditions.
        \item A quantitative bound for the Green-Tao notion of rank for polynomials, significantly improving on a result of Karam.
    \end{itemize}
\end{abstract}

\section{Introduction}

Regularity lemmas for polynomials reduce the analysis of polynomials in many variables to the analysis of a bounded number of approximately independent polynomials.
They have served as a key component in a variety of important results, as in the proof of the finite-field Gowers inverse conjecture in additive combinatorics and higher-order Fourier analysis (e.g.~\cite{GreenTao09,TaoZi12}),
Stillman's conjecture in commutative algebra~\cite{AnanyanHo20}, 
and the Reed-Muller weight distribution problem in coding theory~\cite{BhowmickLo15,KaufmanLoPo12},
to name but a few.
They are a key component in higher-order arithmetic regularity lemmas (e.g., \cite{GowersWo11, Green06, GreenTao09, HatamiLo11, KaufmanLo08}).\footnote{Roughly speaking, arithmetic regularity lemmas approximate general functions in terms of low-degree polynomials that are approximately independent.
(This approximate independence is indispensable; 
to quote Green~\cite{Green06}: second-order arithmetic regularity lemmas ``reduce the study of general functions...\ to the study of projections...\ onto quadratic factors. This, however, is of little consequence unless we can study those supposedly simple objects.'')}
As is usually the case with regularity lemmas, they unfortunately only have weak bounds, 
which are typically inherited by applications.
In this paper we introduce a weaker regularity lemma for polynomials, with far better bounds, that is nevertheless powerful enough for some applications.
For example, our results yield a method to prove \emph{upper} bounds for arithmetic circuits, as well as give strong bounds for a result of Karam~\cite{Karam23} on the relation between the Green-Tao rank of a polynomial and its range. 


Roughly speaking, the usual regularity lemmas for polynomials (see, e.g., Proposition 8.1 in~\cite{ErmanSaSn19}, Lemma~2.4 in~\cite{GreenTao09}, or Lemma 3.11 in \cite{Green06})
can be described as results about polynomial decompositions. Namely, they express any $n$-variate polynomial $P\colon \F^n \to \F$---and more generally, any polynomial map $\P \colon \F^n \to \F^m$---in terms of a bounded number (independent of $n$) of polynomials $\X=(X_1,\ldots,X_k)$, that is, $P=F(\X)$, with the property that the polynomials in $\X$ ``behave'' approximately like $k$ (jointly independent) variables.
%
Although the precise definition of ``behave'' can differ depending on the context, the intuition is as follows: for many purposes, any polynomial 
can be analyzed as though it had only $k$ variables.
The number $k$ of parts $X_1,\ldots,X_k$ in the decomposition depends only on the degree $d$ and the number $m$ of the given polynomials $\P=(P_1,\ldots,P_m)$; as mentioned above, this dependency---befitting a regularity lemma---can be quite astronomical, often having worse than a tower-type growth (see, e.g., the Remark following Lemma~2.4 in~\cite{GreenTao09}).\footnote{Interestingly, recent regularization results of Lampert and Ziegler~\cite{LampertZi24,Lampert23} give strong bounds; however, unlike regularity lemmas, these do not provide a decomposition of $\P$ as a polynomial in $X_1,\ldots,X_k$, but rather place $\P$ in the ideal generated by them.} 
A main reason for these weak bounds is that in order to guarantee approximate independence, 
the polynomial parts $X_1,\ldots,X_k$, as well as all their linear combinations, are required 
to have \emph{rank} (see Section~\ref{sec:WR}) that is large as a function of $k$, their number.
Thus, the larger the number of parts, the more difficult it is to satisfy the rank requirement, which translates to an extremely large number of steps in the process to construct them.\footnote{See, e.g., Proposition~8.1 in~\cite{ErmanSaSn19} for a clean example of this phenomenon: the reason that the process terminates at all is a decrease in the lexicographic order of the tuple of degrees of $\X$ in each step, even as the tuples get longer and longer.}

In the setting of graphs, Frieze and Kannan~\cite{FriezeKa99} (see also~\cite{CF12,FLZ17}) famously obtained a weak  regularity lemma that has both strong bounds and multiple interesting applications. 
For comparison, the usual graph regularity lemma of Szemer\'edi~\cite{Sze78} gives a decomposition of the edge set---induced by a vertex partition---into
a bounded number of parts (independent of the number of vertices),
with the guarantee that almost all graph parts
are ``approximately'' random.
The weak graph regularity lemma decomposes any graph in a similar sense, 
but comes with only one, global guarantee. 
Our weak regularity lemma for polynomials is partly inspired by this approach.

\subsection{Weak regularity for polynomials}

The basic idea of our weak regularity lemma is to relax the requirement on the $k$ polynomials $X_1,\ldots,X_k$, and only require that at least one of them, of maximal degree, behaves like a free variable with respect to all the others. 
That is, conditioned on $X_2,\ldots,X_k$ taking any given value $y$, we require that $X_1$ takes any value in $\F$ with approximately the same frequency (see Remark~\ref{remark:conditional}). 
This imitates the case of a genuinely free variable; or more concretely, the case where $X_1=x_1$, and $X_i=X_i(x_2,\ldots,x_n)$ for $i>1$, where $x_1,\ldots,x_n$ are variables.

Intuitively, the benefit
of having a weak regular decomposition for a polynomial map $\P$ 
is that it reveals some information about 
the univariate polynomials ``induced'' by $\P$.
In particular, it enables us to argue about the image of $\P$ (for which standard rank-based tools are unsuitable, as they typically give information only about equidistribution);
indeed, a corollary of the weak regularity lemma will be that lines that 
$\Image(\X)$ intersects sufficiently often
must in fact be fully contained in the image of $\X$. This translates to a polynomial curve (univariate polynomial map) contained in the image of $\P$. 
See below for more about applications. 

We note that throughout the paper we work with finite fields $\F$, for which the notion of weak regularity can be most elementarily defined. Moreover, rank-based methods, which we rely on extensively in the proof (as do proofs of other regularity lemmas for polynomials), tend to be most challenging in the finite-field setting. See Remark~\ref{remark:inf-field} for more on the case of infinite fields.

Formally, we define weak regularity as follows.
We refer the reader to Section~\ref{sec:WR} for a more formal discussion, and the proofs.
For tuples of polynomials $\P=(P_1,\ldots,P_m)$ and $\X=(X_1,\ldots,X_k)$ over $\F$, we write $\P \sub \F[\X]$ if every component $P_i$ of $\P$ can be written as $F_i(X_1,\ldots,X_k)$ for some polynomial $F_i$.
Henceforth, for any $k$-tuple $v=(v_1,\ldots,v_k)$, we use the notation $v'$ for the $(k-1)$-tuple $v'=(v_2,\ldots,v_k)$.

\begin{definition*}[weak regularity]
    A tuple of polynomials $\P$ over $\F=\F_q$
    has a \emph{weak $\e$-regular decomposition} of size $k$
    if $\P \sub \F[\X]$, for some $k$ homogeneous polynomials $\X=(X_1,\ldots,X_k)$,
    such that:
    \begin{itemize}
        \item for every $y \in \F^k$,
    $|\Pr[\X=y]-q^{-1}\Pr[\X'=y']| \le \e q^{-k}$,
    \item $\P \nsubseteq \F[\X']$, and
    \item $\deg(X_1) = \max_i \deg(X_i)$.
    \end{itemize}
\end{definition*}

The probabilities here are with respect to a uniformly random input; that is, $\Pr[\X=y]$ is a shorthand for $\Pr_{x \in \F^n}[\X(x)=y]$, with $x$ chosen uniformly at random from $\F^n$.
Intuitively, the probabilistic condition can be motivated as follows. If $X_1$ were truly an independent variable $x_1$, with $X_i = X_i(x_2, \ldots, x_n)$ for $i > 1$, we would certainly have both $\Pr[X_1 = y_1] = \frac{1}{q}$ and $\Pr[\X = y] = \Pr[X_1 = y_1]\Pr[\X' = y']$. Hence, the probabilistic condition mimics the behavior we would expect for an independent variable $x_1$.

Note that the two additional conditions at the end of the definition are necessary to avoid trivialities.
This is clearly the case for the first, which guarantees that the decomposition of $\P$ depends on $X_1$.
As for the other additional condition, that $X_1$ is of maximal degree---note that we can decompose any polynomial $P=P(x_1,\ldots,x_n)$, of degree $d$ in $x_1$ say, 
as $\sum_{i=0}^d x_1^i C_i(x_2,\ldots,x_n)$;
that is, we always have a small decomposition $P=F(x_1,C_0,\ldots,C_{d})$ (where 
$F(z, y_0,y_1,\ldots,y_d) = \sum_{i=0}^d z^i y_i$).
Although $x_1$ is genuinely independent of the other parts of the decomposition $C_0,\ldots,C_d$, its degree $\deg(x_1)=1$ is generally far from the largest degree among them.

Let us also mention two extreme cases in the definition of weak regularity.
First, if $\P$ is a tuple of forms that, viewed as a map $\P \colon \F^n \to \F^m$, is already equidistributed, that is, $|\Pr[\P=y]-q^{-m}| \le \e q^{-m}$ for every $y \in \F^m$,
then $\P$ trivially has a small weak $2\e$-regular decomposition.\footnote{$|\Pr[\P=y]-q^{-1}\Pr[\P'=y']|
\le |\Pr[\P=y]-q^{-m}|+q^{-1}|\Pr[\P'=y']-q^{-(m-1)}| \le 2\e q^{-k}$,
so can take $\X=\P$.}
Second, $\P$ always 
has the trivial decomposition $\P=\P(\X)$ with $\X=(x_1,\ldots,x_n)$ its variables, which is a weak $0$-regular (i.e., a perfectly weak regular decomposition; assuming without loss of generality that $\P$ depends on $x_1$), though of a trivially large size $n$.

Our main result is that a rather small weak regular decomposition always exists (see Theorem~\ref{thm:WRL}).

\begin{theorem}[weak regularity lemma]\label{thm:WRL-into}
    Let $\e=q^{-c}$ with $c > 1$.
    Every polynomial $m$-tuple
    $\P$ over $\F=\F_q$, of degrees at most $d$
    with $d < \ch(\F)$,
    has a weak $\e$-regular decomposition
    of size at most 
    $(mc)^{2^{d(1+o(1))}}$.\footnote{Henceforth, the $o(1)$ term in $d(1+o(1))$ depends only on $d$, going to $0$ as $d$ goes to infinity.}
\end{theorem}

For comparison, the usual regularity lemmas for polynomials only give tower-type bounds or worse; namely, the bounds are at the very least a tower of exponents, of height depending on $d$, with $m$ at the top. Our weak regularity lemma, in particular, is polynomial in $m$ (which is especially relevant in the important case of low-degree polynomials $d=O(1)$).

The proof of our weak regularity lemma relies on several components, including a combinatorial regularization process over ``punctured'' spaces 
of polynomials, 
and a Fourier-analytic quasirandomness argument over such spaces. See Section~\ref{subsec:proof} for more information.

\subsection{Applications for arithmetic circuits}

Using the weak regularity lemma, one can prove results that go beyond standard notions of structure for polynomials such as rank or partition rank. These standard notions only give ``shallow'' decompositions---their building blocks are generally just one degree lower---whereas we will be interested in decompositions that go ``deeper''.
Concretely, in the next two sections we present applications that give novel structural results for polynomials based solely on their image.

Inspired by~\cite{Karam23}, we make the following definition for polynomial maps (that is, tuples of polynomials).
Define the \emph{univariate degree} of a polynomial map $\P\colon \F^n \to \F^m$ 
to be the smallest degree of a non-constant polynomial curve contained in its image: 
\[\udeg(\P)=\min\{\deg(\U) \setmid \U \colon \F\to\F^m \text{ non-constant} \colon \Image\U \sub\Image\P\}.\]
Thus, $\udeg(\P)>u$ means that $\Image(\P)$ contains no non-constant polynomial curve (image of a univariate polynomial map) of degree at most $u$.
Note that $1\le \udeg(\P) \le \deg(\P)$ ($:= \max_i \deg(P_i))$ for any non-constant $\P$\footnote{Indeed, if $\P(u)\neq \P(v)$ then $\U(t)=\P(u+t(v-u))$ is non-constant and $\deg(\U)\le \deg(\P).$} (see also~\cite{Karam23} and Section~\ref{subsec:ud-properties}).
Intuitively, for a single polynomial, univariate degree quantifies (lack of) surjectivity; more generally, $\udeg(\P)$ quantifies the ``complexity'' of the curves found in the image of $\P$, starting from being line-free (or, over $\F_3$, being a cap set) which corresponds to $\udeg(\P)>1$.

We note that univariate degree is weaker than standard notions of variety-evasiveness~\cite{DL12,DKL14}, as it only forbids containment of an entire low-degree curve, rather than large intersections with every such curve. 
Thus, applications requiring large $\udeg(\P)$, as we give in this paper, require an even weaker condition on $\Image(\P)$ than being variety evasive.
We discuss the connection with subspace- and variety-evasive sets in Section~\ref{subsec:ud-properties}.

Let $L(P)$ denote the arithmetic formula size\footnote{That is, the smallest number of leaves in an arithmetic formula computing $P$.} of a polynomial $P$.
It is known that for most degree-$d$ polynomials $P$ in $n$ variables, over any field, $L(P) \ge \binom{n+d}{d}^{1-o(1)}$. Indeed, a nice argument of Hrube\v{s} and Yehudayoff (Theorem~3.7 in~\cite{HY11}) shows that this holds even when restricted to polynomials with coefficients in $\{0,1\}$.
In fact, a similar conclusion\footnote{This traces back to Sto{\ss}~\cite{Sto89}, at least in the univariate case. See also Lovett~\cite{Lov11}.} holds for arithmetic circuit size $S(P)$.
We denote more generally by $S(\P)$ the (multi-output) arithmetic circuit size\footnote{That is, the smallest number of $+$ and $\times$ gates of fan-in $2$ in an arithmetic circuit computing simultaneously all components of $\P$.} of a polynomial map $\P=(P_1,\ldots,P_m) \colon \F^n \to \F^m$.

We use the weak regularity lemma to show that for low-degree polynomials, 
high univariate degree 
implies the existence of formulas and circuits that are considerably smaller than the typical $n^{d-o(1)}$, with a power-saving bound.\footnote{As usual, the $o(1)$ term, here and below, goes to $0$ as $n$ goes to infinity.}

\begin{theorem}\label{thm:formula-ub}
    Let $\F$ be a finite field with $\ch(\F)>d$, and let $d \le \frac12\log_2\log_2(n)$.
    For any $n$-variate polynomial $P$ over $\F$ with $\deg(P) \le d$ and $\udeg(P) = u$,
    \[L(P) \le  n^{\lfloor d/u \rfloor + o(1)}.\]
    Moreover, for any polynomial map $\P \colon \F^n \to \F^m$ with $\deg(\P) \le d$, $\udeg(\P) = u$, 
    and $m \le 2^{\log(n)^{o(1)}}$, 
    \[S(\P) \le n^{\lfloor d/u \rfloor + o(1)}.\]
\end{theorem}


We note that any improvement to the upper bound in the weak regularity lemma would translate to a less stringent condition on $d$ in Theorem~\ref{thm:formula-ub} (e.g., a polynomial upper bound $d^C$ in Theorem~\ref{thm:WRL} would translate to the condition $d \le n^{o(1)}$).

Our machinery also has a direct implication for the study of depth-$4$ arithmetic formulas.
A depth-$4$ $\sum^{[r]}\prod^{[a]}\sum\prod^{[b]}$-formula (e.g.,~\cite{KayalSaSa14,KumarSa14})
for a polynomial $P$ is a decomposition of the form 
$P=\sum_{i=1}^r \prod_{j=1}^{a_i} Q_{i,j}$ with $a_i \le a$ and $\deg(Q_{i,j})\le b$. 
It is homogeneous if each $Q_{i,j}$ is homogeneous and all products $\prod_j Q_{i,j}$ are of the same degree.
An important challenge in algebraic complexity is to understand the interplay among the fan-in parameters $(r,a,b)$. 
Trivially, if the bottom fan-in $b$ is allowed to be as large as $\deg(P)$, then $P$  has a $\sum^{[r]}\prod^{[a]}\sum\prod^{[b]}$-formula with $r=a=1$.
The question becomes interesting when the bottom fan-in $b$ is bounded away from $\deg(P)$.
Using the weak regularity lemma, we 
deduce an \emph{upper bound} on the top fan-in $r$ from a lower bound on the univariate degree. (It is worth noting that the model of depth-$4$ homogeneous arithmetic formulas with top fan-in $r$ was shown by Kumar and Saraf~\cite{KumarSa14} to be already powerful for $r=\Theta(\log n)$.)

\begin{theorem}\label{thm:top-fan-in}
    Let $P \colon \F^n \to \F$ with $\deg(P) \le d$, over any finite field $\F$ with $d < \ch(\F)$.
	Suppose $P \in \F[\P]$ with $\P=(P_1,\ldots,P_m)$, $\deg(\P) \le d$, such that $\udeg(\P) = u$.
	Then there is a 
	$\sum^{[r]}\prod^{[2u]}\sum\prod^{[d/u]}$-formula
	for $P$ with top fan-in $r \le (2m)^{2^{d(1+o(1))}}$. 
    Moreover, the formula is homogeneous 
    if $P$ is.
\end{theorem}

Contrapositively, the above results can also be interpreted as implying a barrier for lower bound results: 
roughly, any strong enough lower bound method must somehow avoid applying to polynomial maps of high univariate degree.
It therefore seems that the notion of univariate degree, and more generally a qualitative understanding of the curves contained in the image of a polynomial map,
could be an important data point to consider when studying questions about arithmetic circuits.\footnote{
    Univariate degree should not be confused with the concept of an elusive function of Raz~\cite{Raz10}, which is a polynomial map $\P$ satisfying $\Image\P\nsubseteq \Image\vec{\Gamma}$, for every polynomial map $\vec{\Gamma}$ with given parameters.	
    Indeed, univariate degree is about containment in the other direction: a lower bound on $\udeg(\P)$ means that $\Image\U \nsubseteq \Image\P$, for every polynomial curve $\U$ of a given degree.}

There are also connections with polynomial identity testing. For example, 
a line of work by Oliveira, Sengupta, and others (e.g.,~\cite{OS24,GOS25}) uses a regularity lemma for polynomials to make  progress towards the goal of obtaining an efficient identity testing algorithm for depth-$4$ circuits with bounded top and bottom fan-ins.
We leave it for future work to determine whether our weak regularity lemma could suffice for these and other applications of regularity in algebraic complexity theory.

\subsection{Applications for generalized ranks of polynomials}

Is there an algebraic explanation for non-surjectivity of (multivariate) polynomials?
An important class of non-surjective polynomials are those of the form $P=F(Q)$ with $1<\deg(Q)<\deg(P)$---sometimes called \emph{uni-multivariate decomposable}~\cite{GutierrezRuSe02}---for some non-surjective univariate $F$ (for example, $P=(Q(x_1,\ldots,x_n))^t$ with $\gcd(t,|\F|-1)>1$). 
For such a polynomial of degree $d$, it must be that $\deg(Q) \le d/2$ (as $\deg(F) \ge 2$).
We show that this is always the case, 
in the sense that every non-surjective polynomial can be written in terms of a rather small number of polynomials all of degree at most $d/2$ (despite most polynomials not being uni-multivariate decomposable~\cite{vonZi15}).
In fact, we also show a significant generalization of this to polynomial maps (note that the image of a polynomial map encodes information about the dependencies between its various polynomials).

Following Green and Tao (\cite{GreenTao09}, Definition~1.6), 
let $\rank_t(P)$, for a polynomial $P$ and $t \ge 0$, denote the least number of 
polynomials of degree at most $t$ such that $P$ can be expressed as a function of only them.
Note that this extends naturally to polynomial maps $\P$.
One easy-to-state application of the weak regularity lemma is that any polynomial that is not surjective, and more generally any polynomial map $\P \colon \F^n\to\F^m$ whose image $\Image\P$ ($\sub \F^m$) does not contain a line,
must have low 
$\rank_t(\P)$, for some $t$ bounded away from 
$\deg(\P)$.

\begin{theorem*}[see Theorem~\ref{thm:surj}]
    Let $\P \colon \F^n \to \F^m$ be a polynomial map with $\deg(\P) \le d$, over any finite field $\F$ with $d < \ch(\F)$.
    If $\Image\P$ does not contain a line then
    \[\rank_{d/2}(\P) \le (2m)^{2^{d(1+o(1))}}.\]
\end{theorem*}

Note that just for the special case of $\F_3$, there is already a vast literature on subsets of $\F_3^m$ avoiding a line, also known as cap sets.
See Remark~\ref{remark:cap-sets} for more on the connection between cap sets and polynomial maps.

One can significantly extend this result by considering the univariate degree $\udeg(\P)$, which we recall is the smallest degree of a curve contained in the image of $\P$. 
We note that $\udeg(\P)=1$ if and only if
$\Image\P$ contains a line (and for $m=1$ this means surjectivity).

It is not hard to show that, as $n$ tends to infinity, we cannot hope to bound $\rank_t(P)$ for $t < \deg(P)/\udeg(P)$ (see, e.g., Section~\ref{subsec:ud-properties}).
Karam proved (see Theorem~1.3\footnote{More precisely, Karam's result is stated for the closely related notion of $\rk_t$; see Section~\ref{sec:Apps}.} in~\cite{Karam23}) that we indeed can for $t=\deg(P)/\udeg(P)$.
\begin{theorem}[\cite{Karam23}]
     Let $P \colon \F^n \to \F$ be a polynomial with $\deg(P) \le d$, over any prime finite field $\F=\F_p$ with $d < p$.
     For $t=\deg(P)/\udeg(P)$, we have $\rank_t(P) \le \gamma(p,d)$, for some function $\gamma$.
\end{theorem}

The proof 
uses a variant of the usual regularity lemma for polynomials, and thus yields an extremely weak bound $\g(p,d)$ on $\rank_t(P)$.

We show that the weak regularity lemma
is sufficiently powerful to bound $\rank_t$ for any polynomial, and in fact also for any polynomial map and over non-prime fields too,
yielding a correspondingly strong bound.
\begin{theorem*}[see Theorem~\ref{thm:rk-bd}]
    Let $\P \colon \F^n \to \F^m$ be a polynomial map with $\deg(\P) \le d$, over any finite field $\F$ with $d < \ch(\F)$.
    For $t=\deg(\P)/\udeg(\P)$, 
    \[\rank_{t}(\P) \le (2m)^{2^{d(1+o(1))}}.\] 
\end{theorem*}

Roughly speaking, our proof follows by showing that for any weak $\e$-regular decomposition $\P=\vec{F}(\X)$, for $\e$ bounded away from $1$, we have $\udeg(\P) \le \deg_1(\vec{F})$, the degree of $\vec{F}$ in its first variable
(so $\Image \U \sub \Image\P$ for some non-constant curve $\U$ with $\deg(\U) \le \deg_1(\vec{F})$).
Using the guarantees of the homogeneity of $\X$ and of the maximality of $\deg(X_1)$ ($=\deg(\X)$), 
and removing ``redundant'' monomials from $\vec{F}$,
we will deduce $\deg_1(\vec{F}) \le \deg(\P)/\deg(\X)$, that is, $\deg(\X) \le \deg(\P)/\udeg(\P) = t$, which would allow us to bound $\rank_{t}$.



\subsection{The proof}\label{subsec:proof}

Towards the proof of the weak regularity lemma, we obtain results on linear spaces of polynomials and forms (homogeneous polynomials) that may be of independent interest. In particular, we analyze subspaces of high-rank polynomials (meaning the rank is allowed to grow with the number of spanning polynomials, i.e., the dimension of the space).
More specifically, we analyze \emph{punctured} spaces of high-rank polynomials. By a punctured space we mean the complement of a linear subspace: $V \sm U$ where $V$ is a linear space (of polynomials) and $U$ is a strict subspace.\footnote{For example, a punctured space of codimension $1$ ($V\sm U$ with $\codim_V(U)=1$) 
is of the form $\{cw + u \setmid c \neq 0, u \in U\}$ for some $w \in V$.}
For comparison, the classical regularity lemmas for polynomials work only with the ``trivial'' punctured space $V\sm \{0\}$; or more concretely, \emph{all} nonzero linear combinations of the parts $\X$ are required to have high rank. 
Obtaining a high-rank punctured space $V\sm U$, 
rather than a high-rank \emph{trivial} punctured space $V\sm\{0\}$,
turns out to be achievable with far better bounds.
Roughly speaking, this is proved by looking for a ``low-rank basis'' in the span of the top-degree parts $X_i$.
If one exists, then we decompose the whole basis simultaneously and replace it by its factors---reducing the maximum degree of the parts $\X$---and repeat. If there is no such basis then the parts span a high-rank punctured space, and we stop.
That  quickly leads to strong bounds.
We call this result a ``rank-regularity lemma'', and note that it has strong bounds since the above process terminates quickly.

The rest of the proof of the weak regularity lemma has, very roughly, the following logical structure (see Section~\ref{sec:WR} for all definitions): 
\begin{align*}
    \X \text{ $t$-rank-regular} &\Leftrightarrow \min_{Q\in \Span(\X')} \rk(X_1+Q) \ge t|\X|\\
    &\Rightarrow \max_{Q\in \Span(\X')} |\bias(X_1+Q)| \le \e q^{-|\X|}\\
    &\Rightarrow \X \text{ weak $\e$-regular.}
\end{align*}
The proof moreover guarantees the additional required conditions, e.g., that $X_1$ has maximal degree.
To translate from high rank to low bias, our proof relies on recent results on the relationship between structure and randomness for polynomials (see e.g.~\cite{Janzer19,Milicevic19,CohenMo21B,MoshkovitzZh22}).
These results are proved by going through the corresponding multilinear forms---which requires the characteristic of the field to be larger than the degree of the parts---by relating analytic rank and partition rank (defined by Gowers and Wolf~\cite{GowersWo11} and by Naslund~\cite{Naslund20}, respectively).

\subsection{Organization}

In Section~\ref{sec:WR} we state and prove our weak regularity lemma. 
The proof has several components. In Section~\ref{subsec:rankularity} we prove a rank-regularity lemma, which finds a small decomposition where every linear combination outside of a strict subspace has high rank.
In Section~\ref{subsec:bias-regularity} we show that a similar condition in terms of bias (rather than rank) suffices to get weak regularity, using Fourier-analytic techniques and known structure-vs-randomness results.
Finally, in Section~\ref{subsec:together}, we combine the aforementioned pieces to prove the weak regularity lemma.
In Section~\ref{sec:Apps} we obtain the applications of the weak regularity lemma outlined above.
We end with open problems in Section~\ref{sec:Open}.

\subsection{Preliminaries}

An important aspect of regularity lemmas for polynomials is that they are independent of the number of variables $n$ (similarly to how, for example, the graph regularity lemma is independent of the number of vertices of the graph). 
Henceforth, let $n \ge 0$ be an (unbounded) integer, let $x_1,\ldots,x_n$ be variables, and denote $\Poly(\F) = \F[x_1,\ldots,x_n]$ (suppressing the dependence on $n$ when no confusion can arise).
%
%
While we work with (formal) polynomials in this paper, note that 
evaluation gives a bijection between functions $\F^n \to \F$ 
and reduced polynomials in $\Poly(\F)$, meaning with individual variable degrees below $|\F|$.
Write $\Poly_d(\F) \sub \Poly(\F)$ for the set of polynomials of (total) degree at most $d$.
We henceforth omit $\F$ from our notation when it is clear from context.
Naturally, the Cartesian power $\Poly^m = \{(P_1,\ldots,P_m) \setmid \forall i \colon P_i \in \Poly\}$ stands for the set of $m$-tuples of polynomials, which induce polynomial maps $\F^n\to\F^m$; and analogously for $\Poly_d^m$.
A form is a homogeneous polynomial. We denote by $\Forms \sub \Poly$ the set of all forms of positive degree.
Denote by $\Forms_d$ the vector space of forms of degree $d$ together with the zero form (by convention, $\deg(0)=-\infty$), 
and by $\Forms_d^m$ the vector space of $m$-tuples of such forms. 
For polynomials $X_1,\ldots,X_k \in \Poly(\F)$, we denote by $\F[X_1,\ldots,X_k]$ the algebra they generate, that is, 
$\{F(X_1,\ldots,X_k) \setmid F \in \F[y_1,\ldots,y_k] \}$,
where $y_1,\ldots,y_k$ are fresh variables.
We denote by $\deg_i(P)$ the degree of $P \in \F[y_1,\ldots,y_k]$ in the variable $y_i$ (that is, its degree when viewed as a univariate polynomial in $y_i$).

We will frequently use boldface letters to emphasize that an object is a tuple; so for example, we will use $\P$, instead of $P$, when dealing with a tuple of polynomials (or a polynomial map).
We will use the convention that for a tuple $\X=(X_1,\ldots,X_k)$ and a set $A$, the notation $\X \sub A$ is a shorthand for $X_i \in A$ for all $i$.
For a tuple $\P \in \Poly^k$, we denote $\deg(\P) = \max_i \deg(P_i)$;
more generally, for a set $V \sub \Poly$ we denote $\deg(V) = \sup_{P \in V} \deg(P)$.

Throughout, probabilities, images, fibers and surjectivity for a polynomial map always refer to the induced function.
Probabilities, unless otherwise specified, are to be understood with respect to a uniformly random input; that is, for $\X \colon \F^n \to \F^k$, $\Pr[\X=y]$ is a shorthand for $\Pr_{x \in \F^n}[\X(x)=y]$ with $x$ chosen uniformly at random from $\F^n$.
The image (or range) of a polynomial map $\P \colon \F^n \to \F^m$ is denoted $\Image\P = \{\P(x) \setmid x \in \F^n\}$ ($\sub \F^m$). For a subset $S \sub \F^n$, we write $\P(S) = \{\P(x) \vert x \in S\}$.

Note that $\Poly_d(\F)$ is a linear space over $\F$.
We write $U \le V$ when $U$ is a subspace of $V$.
We denote the linear span of $\X \in \Poly^k(\F)$ by 
$\Span\X = \Span\{X_1,\ldots,X_k\} = \{a \cdot \X \setmid a \in \F^k\}$, 
which is of course a subspace of $\Poly(\F)$.
Note that if $\X \in \Poly_d^k(\F)$ then $\Span\X \le \Poly_d(\F)$.

For $\P \in \Poly^m(\F)$, we refer to an expression of the form $\P=\vec{F}(\X)$, or $\P = \vec{F} \circ \X$, 
for polynomial maps $\vec{F} \colon \F^k \to \F^m$ 
and $\X \in \Poly^k(\F)$, 
as a \emph{decomposition} of $\P$ (so each component $P_i$ of $\P$ is decomposed $P_i=F_i(\X)$).
Note that this is a composition of (formal) polynomials,
and moreover, the choice of $\vec{F}$ need not be unique as $X_1,\ldots,X_k$ need not be independent.
When the map $\vec{F}$ is immaterial, we denote the decomposition by $\P \sub \F[\X]$.

As mentioned above, all fields considered in this paper are finite; we denote by $\F_q$ the finite field of size $q$ (so $q$ is a power of the characteristic $\ch(\F)$).

\section{Weak Regularity}\label{sec:WR}

For convenience, we repeat here the definitions of weak regularity and the weak regularity lemma.
Recall that for a $k$-tuple $\y=(y_1,\ldots,y_k)$, we denote by $\vec{y}'$ the $(k-1)$-tuple $\y'=(y_2,\ldots,y_k)$.

\begin{definition}[weak $\e$-regularity]\label{def:reg}
    Let $\F=\F_q$.
    We say that $\X \in \Forms^k(\F)$ is \emph{weak $\e$-regular} if $\deg(X_1)=\deg(\X)$,
    and for every $y \in \F^k$,
    \[\big| \Pr[\X=\y] - q^{-1}\Pr[\X'=\y'] \big| \le \e q^{-k}.\]   
    For $\P \sub \Poly(\F)$, we say that $\P \sub \F[\X]$ is a weak $\e$-regular \emph{decomposition} if 
    $\X \in \Forms^k(\F)$ is weak $\e$-regular
    and $\P \nsubseteq \F[\X']$.\footnote{By convention, the empty decomposition ($k=0$) is  weak $\e$-regular for any $\e$. Thus, constant polynomial maps $\P$ trivially have a weak $\e$-regular decomposition.}
\end{definition}
Roughly speaking, a tuple of forms $\X$ is weak $\e$-regular if at least one of its members---namely, $X_1$---is of largest degree, and is ``unbiased'' and ``independent'' of the rest.
Moreover, a polynomial map $\P$ has a weak $\e$-regular decomposition if it can be expressed in terms of a weak $\e$-regular $\X$, and depends on $X_1$.
Intuitively, a weak regular decomposition of $\P$ allows us to think of $\P$ as having only $k$ variables $X_1,\ldots,X_k$, and although these may not be independent, free variables, at least one of them approximately is.

Let us also note that the probabilistic condition in Definition~\ref{def:reg} means that $\Pr[\X=\y]/q^{-k} = \a \pm \e$ for every $\y \in \F^k$ with $\Pr[\X'=\y']/q^{-(k-1)} = \a$ (recall $\X'$ is a $(k-1)$-tuple).
Intuitively, $\X$ is distributed like $\F_q \times \X'$, up to an additive error of $\e$.
%
We also mention that Definition~\ref{def:reg} implies $\deg(\X) \le \deg(\P)$ (see Corollary~\ref{coro:deg-X}).

\begin{theorem}[weak regularity lemma]\label{thm:WRL}
    Let $c \ge 1$.
    Every 
    $\P \in \Poly_d^m(\F_q)$,
    with $d < \ch(\F_q)$,
    has a weak $\e$-regular decomposition
    of size 
    \[k \le (2mc)^{2^{d(1+o(1))}},\]
    where we write $\e q^{-k} = q^{-ck}$ for the error term.
\end{theorem}

Note that the decomposition size is polynomial in $m$.
We remark that Theorem~\ref{thm:WRL} is even stronger than the statement in the Introduction (Theorem~\ref{thm:WRL-into}), since here $\e$ is a function of the decomposition size $k$ (namely, $\e=q^{-(c-1)k}$). 

\begin{remark}\label{remark:conditional}
One can phrase weak regularity in terms of conditional probabilities, by taking $\e$ small enough.
Indeed, since $\deg(\X) \le d$ (Corollary~\ref{coro:deg-X}),
for every $\y \in \F^{k}$ we have that $\Pr[\X'=\y'] \ge q^{-(k-1)d}$ if nonzero, by Warning's second theorem, 
in which case for any $C\ge 1$,
\[\big|\Pr[\X=\y \vert \X'=\y'] - q^{-1}\big| 
\le \e q^{-k}/\Pr[\X'=\y'] \le q^{-ck + (k-1)d} \le q^{-Ck} ,\]
by taking $c=C+d$ in Theorem~\ref{thm:WRL} (which leaves unchanged the shape of its bound on $k$).
\end{remark}


\subsection{Rank-regularity}\label{subsec:rankularity}

The proof of the weak regularity lemma has several components.
At the core of the proof is what we call a \emph{rank-regularity lemma}\footnote{Or ``rankularity'' lemma for short.} (see Definition~\ref{def:rankularity}). 
Roughly, the requirement for the parts $\X$ in a rank-regular decomposition
is that they are homogeneous (not necessarily all of the same degree), 
and that every linear combination of $\X$ outside of a strict subspace has high rank (see Definition~\ref{def:rank}).
We will later 
convert a rank-regular decomposition into a weak regular decomposition. 
We begin by setting up the basic definitions we need, 
and state the results that form the steps of our overall strategy.

An important definition in the study of polynomials is that of \emph{rank}, sometimes also called strength~\cite{AnanyanHo20} or Schmidt rank~\cite{Schmidt85}.
Recall that a polynomial is \emph{reducible} if it can be factorized as the product of two non-constant polynomials (all polynomials over the same field).
\begin{definition}[rank]\label{def:rank}
    The \emph{rank} of a nonzero form  $P\in \Forms_d(\F)$ is:
    \begin{itemize}
        \item For $d \ge 2$:
        the least number of reducible forms that sum up to $P$:
        \[\rk(P) = \min\big\{ r\, \big\vert \exists R_1,\ldots,R_r \in \Forms_d(\F) \text{ reducible} \colon P = R_1+\cdots+R_r\big\}.\]
        \item For $d = 1$, $\rk(P)=\infty$. 
    \end{itemize}
    More generally,
    the rank of a polynomial $P \in \Poly(\F)$ is the rank of its top-degree homogeneous part,\footnote{If $\deg(P)=d$ and we write $P=\sum_{i=0}^{d} P_i$ with $P_i$ a form of degree $i$, then $\rk(P)=\rk(P_d)$.}
    or $0$ if $P$ is a constant. 
\end{definition}

Note that for $P \in \Forms_d$ with $\rk(P) \le r$ ($<\infty$), we can write $P = Q(Y_1,\ldots,Y_{k})$ with $k \le 2r$, for some quadratic form $Q$ and forms $Y_i \in \Forms$ with $\deg(Y_i) <d$.
This uses the fact that a reducible form $R$ factorizes as the product of two (lower-degree) forms; indeed, the product of the lowest-degree homogeneous parts of the two factors, and the product of the highest-degree parts, are both homogeneous parts of $R$ and thus must be of the same degree.

The core of the proof of our weak regularity lemma is a weak regularity lemma for $\rk$.
For $\P \sub \Poly$, call a decomposition $\P \sub \F[\X]$ \emph{minimal} if 
$\P \nsubseteq \F[\X^*]$ for every $\X^* \sub \Forms$ with $\Span\X^* \subsetneq \Span\X$. Put differently, 
$\P \nsubseteq \F[W]$ for every homogeneous subspace $W \lneq \Span\X$.


\begin{definition}[rank-regularity]\label{def:rankularity}
For any $t \ge 0$:
\begin{itemize}
    \item $\X \in \Forms^r$ is \emph{$t$-rank-regular} if $\rk(X) \ge tr$ for all $X \in \Span\X$ outside of a strict subspace.\footnote{By convention, the empty tuple ($r=0$) is $t$-rank-regular.}
    \item $\P\in\Poly^m(\F)$ has a $t$-rank-regular decomposition of size $r$ 
    if $\P \sub \F[\X]$ 
    is a minimal decomposition for some $t$-rank-regular $\X \in \Forms^r(\F)$.
\end{itemize}
\end{definition}

Note that ``all $X \in \Span\X$ outside of a strict subspace'' refers, explicitly, to all $X \in (\Span\X)\sm U$ for some strict subspace $U \lneq \Span\X$.\footnote{For example, $U$ might be $\Span(\X')$, provided $\X$ is linearly independent.}
In other words, $\X$ is $t$-rank-regular if $\Span\{X \in \Span\X \setmid \rk(X) < tr\} \neq \Span\X$.
This aspect\footnote{Equivalently, this is about high-rank punctured spaces: if homogeneous $V \le \Poly$ has a punctured space $W$ ($= V\sm U$, $U \lneq V$) with $\rk(P) \ge t\dim(V)$ for every $P \in W$, 
then any basis of forms for $V$ is $t$-rank-regular (and vice versa).}
of the definition is one reason that we are able to obtain strong bounds.
We also mention that Definition~\ref{def:rankularity} implies $\deg(\X) \le \deg(\P)$ (see Corollary~\ref{coro:deg-X}).

\begin{theorem}[rank-regularity lemma]\label{theo:rankular}
    Every $\P \in \Poly^m_{d}$
    has a $t$-rank-regular decomposition 
    of size at most $((2t+1)dm)^{2^d}$.
\end{theorem}

\subsubsection{Linear spaces of polynomials}

Here we prove several useful properties of linear spaces of polynomials, which will be used in the proof of the rank-regularity lemma and the weak regularity lemma.
We begin with some useful notation.

Henceforth, for a linear space $V \le \Poly$ and $R>0$, we denote the span of its low-rank polynomials by
\[U_R(V)=\Span\{X \in V \setmid \rk(X) < R\}.\]
Note that $\X \in \Forms^r$ ($r \ge 1$) is $t$-rank-regular if and only if $U_R(V) \neq V$ for $V=\Span\X$ and $R=tr$.

We say that a finite-dimensional linear space $V \le \Poly$ is \emph{homogeneous} if $V$ is spanned by forms.\footnote{By convention, the zero space $\{0\}$ is homogeneous.}
Equivalently, $V$ is the sum---and thus the direct sum---of its homogeneous parts: 
$V = \bigoplus_{i \ge 0} V_i$ where $V_i := V \cap \Forms_i$.
We abbreviate $V^\uparrow = V_{\deg(V)}$, that is, 
\[V^\uparrow = V \cap \Forms_{\deg(V)} ,\]
where we recall the notation $\deg(V) = \max_{X \in V} \deg(X)$.

For the proof of the rank-regularity lemma (Theorem~\ref{theo:rankular}) we will need a result 
on rank-regularity of forms of top degree (Corollary~\ref{coro:duplicate-rankularity} below). It is an immediate consequence of the following structural result about $U_R(V^\uparrow)$.

\begin{lemma}\label{lemma:U-top}
    Let $V \le \Poly$, $R>0$. If $V$ is homogeneous then 
    $U_R(V)\cap V^\uparrow = U_R(V^\uparrow)$.
\end{lemma}
\begin{proof}
    We need to show $U_R(V)\cap V^\uparrow \sub U_R(V^\uparrow)$, as the reverse containment trivially holds.
    Let $P \in U_R(V)\cap V^\uparrow$, so that $P \in V$ is a form with $\deg(P)=\deg(V)$, 
    and $P$ is a linear combination of $P_i \in V$ with $\rk(P_i) < R$.
    Then $P$ is a linear combination of their homogeneous parts of degree $\deg(V)$, which are:
    \begin{itemize}
        \item also in $V$ (as $V$ is homogeneous), and thus in $V^\uparrow$, and
        \item of rank $<R$ (by definition of rank for non-homogeneous polynomials).
    \end{itemize}  
    Therefore $P \in U_R(V^\uparrow)$, so we are done.
\end{proof}

For a tuple $\Y=(Y_1,\ldots,Y_m) \in \Forms^r$, denote its set of top-degree elements by $\Y^\uparrow = \{Y_i \setmid \deg(Y_i) = \deg(\Y)\}$.
\begin{corollary}\label{coro:duplicate-rankularity}
    Let $\Y \in \Forms^r$, and let $\Y^* \in \Forms^r$ 
    have the same underlying set as $\Y^\uparrow$.
    For any $t$, if $\Y^*$ is $t$-rank-regular then $\Y$ is $t$-rank-regular.
\end{corollary}
\begin{proof}
    Put $V=\Span(\Y)$, $V^* = \Span(\Y^*)$, and note $V^*=\Span(\Y^\uparrow) = (\Span\Y)^\uparrow = V^\uparrow$.
    Let $R=tr$, and suppose $U_R(V)=V$. We need to show $U_R(V^*)=V^*$ (with the same $R$, as the tuples $\Y$ and $\Y^*$ have the same number $r$ of polynomials), or equivalently, $U_R(V^\uparrow)=V^\uparrow$.
    This holds by Lemma~\ref{lemma:U-top}: $U_R(V^\uparrow) = U_R(V) \cap V^\uparrow = V \cap V^\uparrow =  V^\uparrow$, so we are done.   
\end{proof}

For the rest of the proof of the weak regularity lemma (Theorem~\ref{thm:WRL}),
we will need Lemma~\ref{lemma:U-hom} and Lemma~\ref{coro:top-rank-reg} below.

We begin with the following closure property of the subspaces $U_R(V)$.

\begin{claim}\label{lemma:low-deg}
    Let $V \le \Poly$, $R>0$, $\dim(V)<\infty$.
    Every $X \in V$ with $\deg(X)<\deg(U_R(V))$ satisfies $X \in U_R(V)$.
\end{claim}
\begin{proof}
    Abbreviate $U=U_R(V)$.
    Note that there exists $X_0 \in U$ with $\rk(X_0)<R$ of maximal degree, $\deg(X_0)=\deg(U)$.
    We claim that $X_0+X \in U$.
    Indeed, since $\deg(X) < \deg(U) = \deg(X_0)$, we have
    $\rk(X_0+X) = \rk((X_0 + X)^\uparrow) = \rk(X_0^\uparrow) = \rk(X_0) < R$.
    Thus, as desired, $X \in U$, being the difference $X=(X_0+X)-X_0$ of two members of the linear space $U$.
\end{proof}

\begin{lemma}\label{lemma:U-hom}
    Let $V \le \Poly$, $R>0$. If $V$ is homogeneous then so is $U_R(V)$.
\end{lemma}
\begin{proof}
Put $U=U_R(V)$. 
For $X \in \Poly$, denote by $X^{(i)} \in \Forms_i$ the degree-$i$ part of $X$ (so that $X = \sum_i X^{(i)}$).
First, we prove another closure property of $U$ (in a homogeneous $V$): for every $X \in V$, if $\rk(X)<R$ then $X^{(i)} \in U$ for every $i$.
Let us prove this.
If $i=\deg(X)$ (that is, $X^{(i)} = X^\uparrow$)
then $\rk(X^\uparrow)=\rk(X) < R$; since $V$ is homogeneous, $X^\uparrow \in V$, so $X^\uparrow \in U$, as needed.
On the other hand, if $i<\deg(X)$
then $\deg(X^{(i)})<\deg(X) \le \deg(U)$, since $X \in U$, 
and since $X^{(i)} \in V$ by the homogeneity of $V$,
we can apply Claim~\ref{lemma:low-deg} on $X^{(i)}$
to deduce $X^{(i)} \in U$.
This proves our claim.

Let $\X = \{X \in V \vert \rk(X)<R\}$, so that $U=\Span(\X)$ by definition.
By the claim above, $\X^* := \{X^{(i)} \setmid X \in \X, i \le \deg(X)\}$ satisfies $\X^* \sub U$.
Since $U = \Span(\X) \sub \Span(\X^*) \sub U$, we deduce that $U$ has a spanning set of forms, which completes the proof.
\end{proof}

\begin{lemma}\label{coro:top-rank-reg}
    Let $V \le \Poly$, $R>0$.
    If $V$ is homogeneous, then $U_R(V)=V$ if and only if $U_R(V^\uparrow)=V^\uparrow$.
\end{lemma}
\begin{proof}
    The forward direction is equivalent to Corollary~\ref{coro:duplicate-rankularity}: if $U_R(V)=V$, then $\deg(U_R(V))=\deg(V)$, 
    which implies $V^\uparrow \sub U_R(V)^\uparrow \sub U_R(V^\uparrow)$, as needed.

    For the other direction, suppose $U_R(V^\uparrow) = V^\uparrow$.
    Then $V^\uparrow \sub U_R(V)$, as $U_R(V^\uparrow) \sub U_R(V)$, 
    and in particular, $\deg(U_R(V))=\deg(V)$.
    Let us denote $V' = \bigoplus_{i<\deg(V)} V_i$. 
    Claim~\ref{lemma:low-deg} says $V'\sub U_R(V)$.
    Then 
    $V = V^\uparrow + V' \sub U_R(V) + U_R(V) = U_R(V) \sub V$,
    that is, $U_R(V)=V$, as needed.
\end{proof}

\subsubsection{Proof of the rank-regularity lemma}

Theorem~\ref{theo:rankular} will be proved using a repeated application of the following decomposition lemma.
It shows that if a tuple of forms $\X$, all of the same degree $d$, is not $t$-rank-regular, then 
we can decompose it into not-too-many forms of varying degrees, all strictly smaller than $d$.
Crucially, the lemma considers not just individual low-rank forms, but the whole subspace they span.

\begin{lemma}\label{lemma:reg-step-eff}
    If $\X \in \Forms_d^r(\F)$ is not $t$-rank-regular,
    then $\X\sub\F[\Y]$ for some $\Y \in \Forms^R$ with $\deg(\Y)<d$ and $R \le 2tr^2$.
\end{lemma}
\begin{proof}
    Assume $d\ge 2$, as otherwise the hypothesis is vacuous.
    Put $V=\Span\X$ ($\leq \Forms_d$) and $k = \dim V$ ($\le r$).
    Let \[\vec{X^-} = \{X \in V \setmid \rk(X) < tr\}.\]
    Suppose $\X$ is not $t$-rank-regular, so that $\Span(\vec{X^-}) = V$.
    Thus, $\vec{X^-}$ contains a basis $\B=(B_1,\ldots,B_k) \in \Forms^k_d$ for $V$.
    This has the following implications:
    \begin{itemize}
        \item Since $\X \sub \Span\vec{B}$ ($=V$), there is a linear map $\vec{L} \colon \F^k \to \F^{r}$ such that $\X = \vec{L}(\vec{B})$.
        \item Since $\vec{B} \sub \vec{X^-} \cap \Forms$, 
        and by the definition of $\rk$ (recall Definition~\ref{def:rank} and the subsequent remark), we have
        $\vec{B} = \vec{Q}(\vec{Y})$
        for some homogeneous quadratic $\vec{Q} \colon \F^R \to \F^k$
        and some $\vec{Y} \in \Forms^R$ with $\deg(\Y)<d$,
        such that $R \le \sum_{i \in [k]} 2\rk(B_i) < 2tr \cdot k$.
    \end{itemize}
    Combining the above, we obtain the decomposition
    $\X = \vec{L}(\vec{Q}(\vec{Y})) = \vec{F}(\Y)$, 
    where $\vec{F} := \vec{L}(\vec{Q}) \colon \F^R \to \F^r$ is homogeneous and quadratic,
    and $\vec{Y} \in \Forms^R$ with $\deg(\Y) < d$, such that $R < 2tr \cdot k \le 2tr^2$. This completes the proof.
\end{proof}

We prove the rank-regularity lemma by an iterative construction with at most $d$ steps, maintaining at all times a tuple of forms of various degrees.
At each step, we move to a minimal decomposition, 
and if it is not a $t$-rank-regular decomposition
then, we show, its subset of forms of maximal degree is also not $t$-rank-regular. This last deduction requires care---as the rank requirement in Definition~\ref{def:rankularity} ($\rk(X) \ge tr$) depends on the number of polynomials in the decomposition---and follows using Corollary~\ref{coro:duplicate-rankularity}.
Having done that, we decompose the maximum-degree forms in terms of lower degree forms using Lemma~\ref{lemma:reg-step-eff},
which strictly decreases the maximum degree of the forms in the decomposition, as needed.

\begin{proof}[Proof of Theorem~\ref{theo:rankular}]
    Assume $\P$ is non-constant, as otherwise there is nothing to prove.
    Let $\vec{X^0} \in \Forms^{r_0}$ be obtained from $\vec{P} \in \Poly^m$ by 
    collecting all nonzero positive-degree homogeneous parts of the components of $\vec{P}$.
    Note that $\vec{P} \sub \F[\vec{X^0}]$ and $r_0 \le dm$. 
    %
    
    We claim that for some $s \ge 0$, there are $s$ homogeneous polynomial maps $\vec{X^1}, \ldots, \vec{X^s}$, with $\vec{X^i} \in \Forms^{r_i}$ for $i=1,\ldots,s$, 
    such that $\P \in \F[\vec{X^s}]$ is a $t$-rank-regular decomposition, 
    and for every $0 \le i \le s$:
    \begin{itemize}
        \item $\P \sub \F[\vec{X^{i}}]$,
        \item $\deg(\vec{X^i}) \le d-i$, and
        \item $r_i \le (2t+1)^{2^i - 1}(dm)^{2^i}$.
    \end{itemize}
    The process terminates (that is, $s$ is well defined);
    indeed, any $\X \sub \Forms$ with $\deg(\X)=1$ is $t$-rank-regular for any $t$, since $\rk(X)=\infty$ for every $X \in \Span(\X)\setminus\{0\}$, 
    so a minimal decomposition $\P \sub \F[\X]$ with $\deg(\X) = 1$ is a $t$-rank-regular decomposition for any $t$.
    Proving the above claim would complete the proof,
    since $1 \le \deg(\vec{X^{s}}) \le d-s$ implies $s < d$, 
    and thus $\vec{X^s} \in \Forms^{r_s}$ with $r_s \le ((2t+1)dm)^{2^d}$, 
    as desired.

    For each $i$ starting with $i=0$, 
    given $\vec{X^{i}}$ we proceed as follows:
    \begin{itemize}
        \item 
        Let $\vec{Y} \in \Forms^{r'}$ be a basis of a homogeneous subspace $W \le \Span(\vec{X^i})$ of minimal dimension satisfying $\P \sub \F[W]$.
        %
        Then $\P \sub \F[\vec{Y}]$ is a minimal decomposition,
        and moreover, $\deg(\vec{Y}) \le \deg(W) \le \deg(\vec{X^i})$,
        and $r' = \dim W \le \dim \Span(\vec{X^i}) \le r_i$.

        \item Now, if $\P \sub \F[\vec{Y}]$ is $t$-rank-regular 
        then set $s=i$, $\vec{X^{s}}=\Y$ and $r_s=r'$,
        finishing the construction.
        Otherwise, $\vec{Y}$ must not be $t$-rank-regular.
        We will use $\vec{Y}$ to construct $\vec{X^{i+1}}$ with $\deg(\vec{X^{i+1}})<\deg(\vec{X^{i}})$.
        In order to apply Lemma~\ref{lemma:reg-step-eff}, we need a tuple of forms satisfying two conditions: they are all of the same degree, and the tuple is not $t$-rank-regular.

        Put $\ell = \deg(\Y)$, and obtain $\vec{Y^*} \in \Forms^{r'}$ from $\Y$ by replacing each component of $\Y$ of degree smaller than $\ell$ by an arbitrary degree-$\ell$ polynomial from $\Y$. 
        By Corollary~\ref{coro:duplicate-rankularity},  $\vec{Y^*}$ is not $t$-rank-regular.

        Apply Lemma~\ref{lemma:reg-step-eff} on $\vec{Y^*}$.
        Merge the resulting tuple of forms with the forms in $\vec{Y}$ of degree smaller than $\ell$ to obtain $\vec{X^{i+1}}\in \Forms^{r_{i+1}}$
        with $\vec{Y} \sub \F[\vec{X^{i+1}}]$, 
        $\deg(\vec{X^{i+1}}) < \deg(\vec{X^i})$,
        and $r_{i+1} \le 2tr_i^2 + r_i \le (2t+1)r_i^2$.
    \end{itemize}
    
    To prove that this construction gives the claim above,
    we proceed by induction on $i \ge 0$.
    The induction base $i=0$ follows trivially from the construction of $\vec{X^0}$, as $\P \sub \F[\vec{X^0}]$ with $\deg(\vec{X^0}) \le d$, and $r_0 \le dm = (2t+1)^0 (dm)^1$.
    For the induction step
    we have, by the construction of $\vec{X^{i+1}}$ and the induction hypothesis,
    that $\P \sub \F[\vec{Y}] \sub \F[\vec{X^{i+1}}]$,
    that $\deg(\vec{X^{i+1}}) \le \deg(\vec{X^i}) -1 \le d-(i+1)$,
    and that 
    $r_{i+1} \le (2t+1)r_{i}^2 \le (2t+1)((2t+1)^{2^{i} - 1}(dm)^{2^{i}})^2 = (2t+1)^{2^{i+1}-1} (dm)^{2^{i+1}}$,
    as needed.
    This completes the proof.
\end{proof}

\subsection{Bias and weak regularity}\label{subsec:bias-regularity}

We will use the notion of bias of a polynomial, which is a standard analytic measure related to deviation from equidistribution.
Henceforth, fix any nontrivial additive character\footnote{A nontrivial additive character of $\F$ is a homomorphism $\chi \colon \F \to \C^\times$ ($\chi(a+b)=\chi(a)\chi(b)$) with $\chi \not\equiv 1$.
Equivalently, $\chi(y) = \omega^{\Tr(cy)}$, with $\omega=e^{2\pi i/p}$, 
$p=\ch(\F)$, $c \in \F^\times$, and $\Tr$ is the trace of $\F$ on $\F_p$.} $\chi \colon \F\to\C$ (for concreteness, if $\F=\F_p$ is a prime field, take $\chi(y) = \omega^y$, where $\omega$ is the principal root of unity, $\omega=e^{2\pi i/p}$, and $y$ on the right hand side is viewed as an element of $\{0,\ldots,p-1\}$).
The \emph{bias} of $P \in \Poly(\F)$ is the complex number
\[\bias(P) = \Exp_{x \in \F^n} \chi(P(x)).\]
As we will see below, simple analysis relates bias and equidistribution as follows (see also e.g.~\cite{BGY20}, or Claim~33 in \cite{BogdanovVi10}):
\[\forall y\in \F\colon \Big|\Pr[P=y]-\frac{1}{|\F|}\Big| \le \max_{a \in \F^\times}|\bias(aP)|.\]

Towards proving that rank-regularity implies weak regularity, we will use the best bound known for 
the structure versus randomness problem for polynomials, 
which is almost linear.
\begin{theorem}[\cite{MoshkovitzZh22}, Corollary~8.5 and Remark~8.3]\label{theo:MZ}
    For any $P \in \Poly_d(\F)$ 
    with $\ch(\F)>d$,
    \[\rk(P) \ge r \quad\Rightarrow\quad |\bias(P)| \le |\F|^{-c_d r/\log_2(r+1)}\]
    and $c_d = \Omega_d(1)$, namely $c_d \ge 2^{-d^{1+o(1)}}$.
\end{theorem}
We remark that~\cite{MoshkovitzZh22} even gives that the logarithmic term in Theorem~\ref{theo:MZ} can be removed if $|\F| \ge r^{\Omega_d(1)}$
(namely, the denominator is $\log_{|\F|}(r+1)+1$), but we do not need this.
We also remark that one can replace the use of Theorem~\ref{theo:MZ} in our proof by earlier bounds of Mili\'{c}evi\'{c}~\cite{Milicevic19} or Janzer~\cite{Janzer19}, and still obtain reasonable, albeit somewhat worse bounds in Theorem~\ref{thm:WRL}.

Using bias, a Fourier-analytic argument will help us to deduce weak regularity.
Recall that $\X$ being weak $\e$-regular requires that for every $y \in \F^{k}$,
\[\big| \Pr[\X=\y] - |\F|^{-1}\Pr[\X'=\y'] \big| \le \e |\F|^{-k}.\]
We will bound the left hand side by 
$\max_{a_1 \neq 0} |\bias(a \cdot \X)|$.
Next, we will bound the latter by $\e |\F|^{-k}$, and deduce weak $\e$-regularity, in the following statement.

\begin{lemma}\label{lemma:bias-to-reg}
    Let $U \lneq V \le \Poly(\F_q)$ with $\dim(V)=k$. 
    Suppose $|\bias(X)| \le \e q^{-k}$ for all $X \in V\sm U$.
    Then any basis $\X \in \Forms^k$ of $V$ that contains a basis of $U$, 
    with $X_1 \notin U$ and $\deg(X_1)=\deg(\X)$,
    is weak $\e$-regular.
\end{lemma}
That is, Lemma~\ref{lemma:bias-to-reg} requires of $\X$ that: $\X$ is a basis of $V$, 
$U \sub \Span(\X')$, $X_1 \notin U$, and $\deg(X_1)=\deg(\X)$.

It will be convenient to use a more geometric formulation of weak $\e$-regularity:
for every line $\ell \sub \F^k$ in the direction of $e_1$ (so $\ell=\{te_1+\z \vert t \in \F\}$ for some $\z \in \F^k$),\footnote{Indeed, any $\y \in \ell$ is of the form $\y=(t,\z')$ for some $t \in \F$, so $\Pr[\X \in \ell]=\Pr[\X'=\z']$ and $\z'=\y'$.}   
\begin{equation}\label{eq:line}
    \forall y \in \ell \colon |\Pr[\X=\y] - |\F|^{-1}\Pr[\X \in \ell]| \le \e |\F|^{-k}.  
\end{equation}

\begin{proof}[Proof of Lemma~\ref{lemma:bias-to-reg}]
We will use the standard fact that the distribution $(\Pr[\X=y])_{y \in \F^k}$ can be expressed in terms of biases, 
by using simple Fourier analysis.
Indeed, since\footnote{$\sum_{a \in \F^m} \chi(a \cdot v) = \sum_{a \in \F^m} \prod_{i=1}^m \chi(a_i v_i)
    = \prod_{i=1}^m \sum_{a_i \in \F} \chi(a_i v_i)
    = \prod_{i=1}^m \mathbf{1}(v_i=0)|\F| = \mathbf{1}(v=0)|\F|^m$.} 
\begin{equation}\label{eq:exp-identity}
    \forall v \in \F^m \colon\quad \Exp_{a \in \F^m} \chi(a \cdot v) = \delta[v=0]
    := \begin{cases}
        1 & v = 0\\
        0 & v \neq 0
    \end{cases}\;,
\end{equation}
we have that for every $y \in \F^k$:
\begin{equation}\label{eq:pr-point}
\Pr_x[\X(x)=y] 
= \Exp_x \,\delta[\X(x)-y=0]
= \Exp_x \Exp_{a \in \F^k} \chi(a \cdot (\X(x)-y))
= \Exp_{a \in \F^k} \bias(a \cdot (\X-y)).
\end{equation}

Recall that $\X$ contains a basis for $U$, and $X_1 \notin U$.
We claim that for every $a \in \F^k$ with $a_1 \neq 0$, we have that $a \cdot \X \notin U$. 
Indeed, write $a\cdot \X = a_1X_1 + X'$ with $X' \in \Span(\X')$. Our assumption on $\X$ implies that $U \sub \Span(\X')$, so if $a \cdot \X \in U$ then $a_1X_1 \in U-X' \sub \Span(\X')$, a contradiction as $a_1 \neq 0$.
Using our assumption on $U$, we deduce that
\begin{equation}\label{eq:bias-assumption}
    \forall a \in \F^k \colon \quad a_1 \neq 0 \quad\Rightarrow\quad 
    |\bias(a \cdot \X)| \le \e|\F|^{-k}.
\end{equation}

Put $v=e_1 \in \F^k$.
Let $\ell \sub \F^k$ be a line in the direction of $v$, so that $\ell = \{tv+z \vert t \in \F\}$ for some $z \in \ell$.
We claim that the following equality holds:
\begin{equation}\label{eq:pr-line}
\Pr_x[\X(x) \in \ell] 
= \Exp_{\substack{a \in \F^k\colon\\a \cdot v = 0}}\, \bias(a \cdot (\X-z)).
\end{equation}
Indeed, 
\begin{align*}
    \Pr_x[\X(x) \in \ell]
    &= \sum_{t \in \F} \Pr_x[\X(x) = tv+z]
    = \sum_{t \in \F} \Exp_{a \in \F^k} \Exp_x \chi(a \cdot (\X(x)-(tv+z)))\\
    &= \Exp_{a \in \F^k} \big( \sum_{t \in \F} \chi(t (a\cdot v)) \big) \Exp_x \chi(a \cdot (\X(x)-z))\\
    &= \Exp_{a \in \F^k} \big(|\F|\delta[a\cdot v = 0]\big) \bias(a \cdot (\X-z))
    = \frac{1}{|\F|^{k-1}}\sum_{\substack{a \in \F^k\colon\\a \cdot v = 0}} \bias(a \cdot (\X-z))
\end{align*}
where~(\ref{eq:pr-point}) is used in the second equality, and~(\ref{eq:exp-identity}) is used in the penultimate equality with $m=1$.
Combining the point equality~(\ref{eq:pr-point}) and the line equality~(\ref{eq:pr-line}), we deduce (for $y=z$)
\[\Pr_x[\X(x)=z] 
    = |\F|^{-1}\Pr_x[\X(x) \in \ell]
    + |\F|^{-k}\sum_{\substack{a \in \F^k\colon\\a \cdot v \neq 0}}\, \bias(a \cdot (\X-z)).\]
Now, we bound the magnitude of the rightmost term as follows:
\[|\F|^{-k}\Big|\sum_{\substack{a \in \F^k\colon\\a \cdot v \neq 0}}\, \bias(a \cdot (\X-z))\Big|
\le \max_{\substack{a \in \F^k\colon\\a \cdot v \neq 0}} |\bias(a \cdot (\X-z))|
= \max_{\substack{a \in \F^k\colon\\a \cdot v \neq 0}} |\bias(a \cdot \X)|.\]
Therefore,
\[\Big|\Pr_x[\X(x)=y] - |\F|^{-1}\Pr_x[\X(x) \in \ell]\Big| 
\le \max_{\substack{a \in \F^k\colon\\a \cdot v \neq 0}} |\bias(a \cdot \X)| 
\le \e|\F|^{-k},\]
using~(\ref{eq:bias-assumption}) in the last inequality.
Thus, $\X$ is weak $\e$-regular, as desired.
\end{proof}

\subsection{Putting everything together}\label{subsec:together}

The proof of the weak regularity lemma will rely on 
Theorem~\ref{theo:rankular}, Theorem~\ref{theo:MZ}, and Lemma~\ref{lemma:bias-to-reg}.
Careful work will be required to convert rank-regularity, which gives information only about the top-degree part (recall Definition~\ref{def:rank}) of the polynomials in our linear space, to weak regularity, which takes into account polynomials of all degrees, while still retaining the homogeneity of the parts.

\begin{proof}[Proof of Theorem~\ref{thm:WRL}]
    Put $g(t,d,m)=((2t+1)dm)^{2^d}$.
    We claim that the following holds 
    for every polynomial $P \in \Poly_d(\F)$ and every $1 \le k \le g(t,d,m)$:
    \begin{equation}\label{eq:bias-to-rk}
        |\bias(P)| \ge q^{-ck} \quad\Rightarrow\quad \rk(P) < tk
    \end{equation}
    for some $t = 2^{d^{1+o(1)}}c^{1+o(1)}\log(2m)$.
    Indeed, by Theorem~\ref{theo:MZ}, $\rk(P) \ge tk$ implies that
    $|\bias(P)| \le q^{-c_d tk/\log_2(tk+1)}$ with $c_d \ge 2^{-d^{1+o(1)}}$,
    that is,
    \begin{align*}
        (-\log_q|\bias(P)|)/k
        &\ge \frac{c_d t}{\log_2(tk+1)}
        \ge \frac{c_d t}{\log_2(tg(t,d,m)+1)}\\
        &\ge \frac{c_d t}{(2^d+1)\log_2((2t+1)dm)} .
    \end{align*}
    It remains to show that the latter is greater than $c$, which would imply the desired bound $|\bias(P)| < q^{-ck}$.
    Take $t=\lceil 8DL \rceil$, for $D=c(2^d+1)/c_d$ and $L=\log_2(32Ddm)$.
    Then $t = D^{1+o(1)}\log(2dm)$, implying $t$ has the aforementioned desired asymptotics. 
    Assuming $c_d \le 1$ as we may, we have $D \ge 1$ and $L \ge 5$.
    Thus, $(2t+1)dm \le 32DL \cdot dm \le (32Ddm)^2 = 2^{2L}$, 
    so 
    \[\frac{c_d t}{(2^d+1)\log_2((2t+1)dm)} \ge \frac{c_d \cdot 8DL}{(2^d+1)\cdot 2L} = 4c > c,\]
    as needed.
    

    With $t$ as above, apply Theorem~\ref{theo:rankular} to find a $t$-rank-regular decomposition $\Y \sub \Forms$ of $\P$ of size at most $g(t,d,m)$
    so that $\P \sub \F[\Y]$ is a minimal decomposition and $\Y$ is $t$-rank-regular.
    For simplicity, assume without loss of generality that $\Y$ is linearly independent\footnote{This preserves $\F[\Y]$ and minimality, and $\rk(X) \ge t\cdot\dim\Span(\Y)$ outside of a strict subspace.} (and nonempty, as otherwise there is nothing to prove).
    Put $V= \Span\Y$ and $k=\dim V$ ($>0$).
    To finish, we will use $\Y$ to obtain a weak $\e$-regular decomposition of $\P$ of size 
    \[k \le g(t,d,m) \le (2mc)^{2^{d(1+o(1))}}.\]
    
    Consider the subspace of $V$ given by
    \[U =  \Span\{ X \in V \setmid \rk(X) < tk\},\]
    and the subset of $V$ given by
    \[U_\e = \{ X \in V \setmid |\bias(X)| \ge q^{-ck}\}.\]
    Then:
    \begin{enumerate}
        \item\label{item:U-hom} $U$ is homogeneous. This is by Lemma~\ref{lemma:U-hom} and since $V$ is (it is spanned by $\Y \sub \Forms$).
        \item\label{item:U-arrow} $V^\uparrow \sm U \neq \emptyset$. 
        Indeed, by Lemma~\ref{coro:top-rank-reg}, the $t$-rank-regularity of $\Y$ implies $U_{tk}(V^\uparrow) \neq V^\uparrow$;
        by Lemma~\ref{lemma:U-top}, $V^\uparrow \sm U = V^\uparrow \sm (U \cap V^\uparrow)
        = V^\uparrow \sm U_{tk}(V^\uparrow) \neq \emptyset$.
        \item\label{item:U-bias} $U_\e \sub U$.
        Indeed, for $X \in U_\e$
        we have $\rk(X) < tk$ by~(\ref{eq:bias-to-rk}), 
        using that $k \le g(t,d,m)$ and that $\deg(X) \le \deg(\Y) \le d$ (see Corollary~\ref{coro:deg-X}).
    \end{enumerate}
    Let $X_1 \in V^\uparrow \sm U$ be nonzero, using~\ref{item:U-arrow} (as $V^\uparrow \sm U \neq \emptyset, \{0\}$),
    add to it a homogeneous basis of $U$, using~\ref{item:U-hom},
    and extend this independent set into a homogeneous basis $\X \in \Forms^k$ of $V$.
    Since $U$ satisfies $(V \sm U) \cap U_\e = \emptyset$ by~\ref{item:U-bias}, we have $|\bias(X)| \le q^{-ck}$ for every $X \in V \sm U$.
    Therefore, we can apply Lemma~\ref{lemma:bias-to-reg} on $U \lneq V$, and $\X$ (with $\e=q^{-(c-1)k}$), and deduce that $\X$ is weak $\e$-regular.

    Finally, since $\Span\Y = V = \Span\X$, 
    we have that $\P \sub \F[\Y] = \F[\X]$; 
    moreover, by definition,
    since $\P \sub \F[\Y]$ is a minimal decomposition,
    so is $\P \sub \F[\X]$.
    In particular, $\P \nsubseteq \F[\X']$, since $\X' \sub \Forms$ satisfies $\Span\X' \subsetneq \Span\X$.
    Thus, $\P \sub \F[\X]$ is a weak $\e$-regular decomposition.
    The size of this decomposition is $k$, 
    which, as discussed above, completes the proof.
\end{proof}

\begin{remark}\label{remark:inf-field}
    Although we do not pursue it here, one can try to prove an analogue of our weak regularity lemma for infinite fields. 
    This would require replacing the probabilistic notion of approximate independence in Definition~\ref{def:reg}, 
    say by a geometric or by a commutative algebra analogue (see, e.g.,~\cite{ErmanSaSn19,LampertZi22}), and adapting the proof, including the notion of bias of a polynomial used in Section~\ref{subsec:bias-regularity}, accordingly.
    To retain the applications of Section~\ref{sec:Apps} would, however, require stronger control than algebraic independence, since the latter alone does not quite guarantee that the image contains entire lines.
    For the structure-vs-randomness theorem we rely on (Theorem~\ref{theo:MZ}), a geometric analogue for forms over algebraically-closed and other infinite fields
    has long been known~\cite{Schmidt85} (see also~\cite{KazhdanLaPo24}) with, in fact, a slightly better, linear bound.
\end{remark}

\section{Applications}\label{sec:Apps}

In this section we obtain several results along the theme that one can deduce strong bounds on the algebraic structure of polynomials, and polynomial maps, based solely on their image.
The proofs use the weak regularity lemma.

We begin with the following standard quantification of algebraic structure, which constrains the degrees of the polynomials in the decomposition.
\begin{definition}[$\rk_t$]\label{def:t-rank}
    Call a polynomial \emph{$t$-reducible} if it is the product of polynomials of degree at most $t$.
    Denote, for $P \in \Poly(\F)$ with $\deg(P)=d$,
    \[\rk_t(P) = \min\big\{ r\, \big\vert \exists R_1,\ldots,R_r \in \Poly_d(\F) \text{ $t$-reducible} \colon P = R_1+\cdots+R_r\big\}. \]
    For completeness, set $\rk_0(P)=0$ for constant $P$. 
    More generally, for a tuple $\P \in \Poly^m$ with $\deg(\P)=d$, 
define $\rk_t(\P)$ to be the least number $r$ of $t$-reducible polynomials $R_1,\ldots,R_r \in \Poly_d$ such that $\P \sub \Span\{R_1,\ldots,R_r\}$.
\end{definition}
This definition of $\rk_t(P)$ was used by Karam~\cite{Karam23}, and is closely related to an earlier definition of Green and Tao~\cite{GreenTao09}, $\rank_t(P)$, which expresses $P$ over a finite field as any function of $\rank_t(P)$ polynomials of degree at most $t$. 
Note that $\rank_t(P) \le O(\deg(P)/t)\rk_t(P)$, as each $t$-reducible polynomial can be expressed in terms of $O(\deg(P)/t)$ polynomials of degree at most $t$; thus, proving an upper bound on $\rk_t(P)$, as we will do below, would imply a similar upper bound on $\rank_t(P)$.

\begin{remark}
    If $\rk_t(P)$ is small, for $t$ bounded away from $\deg(P)$, then
    $P$ has an unusually small Hamming distance to some polynomial of considerably smaller degree (unlike most polynomials of a given degree~\cite{BHL12,BGY20}).
    More precisely, by a standard Fourier-analytic argument (e.g.~\cite{Janzer19}), 
    every $P \in \Poly_k(\F_p)$ with $\rank_t(P)=r$ has correlation  
    $|\bias(P-Q)| \ge p^{-r/2}$ ($\ge p^{-O(k\rk_t(P)/t)}$) 
    with some $Q \in \Poly_t(\F_p)$,
    and therefore $P$ agrees with $Q$ on a $\frac{1}{p}(1+\d)$-fraction of the inputs with $\d\ge\Omega(p^{-r/2})$.
    Like the Gowers inverse theorem, this type of result is closely related to polynomiality testing in the low-soundness/list-decoding regime.  
\end{remark}


An application of the weak regularity lemma that is quite simple to state involves polynomial maps whose image is a line-free set. Recall that a line in $\F_q^m$ is a set of $q$ elements of the form $\{u+tv \,\vert\, t \in \F_q\}$ for some $u,v \in \F_q^m$ ($v \neq 0$).

\begin{remark}\label{remark:cap-sets}
    A subset of $\F_3^m$ avoiding a line is known as a cap set. 
    The most basic example of a cap set is $\{0,1\}^m$;
    this cap set leads to the following easy observation:
    for any (low-degree) polynomial map $\P\colon \F_3^n\to\F_3^m$, in any number of variables $n$, 
    its ``square'' $\P^2:=(P_1^2,\ldots,P_m^2) \colon \F_3^n\to\F_3^m$ is a (low-degree) polynomial map with a line-free image. 
    See, e.g.,~\cite{EHPS24} and references within for more constructions of sets free of $3$-term arithmetic progressions.

    Over fields other than $\F_3$, one can also easily construct a low-degree polynomial map whose image is a cap set, meaning it contains no three collinear points (so not just free of $3$-term arithmetic progressions $\{u,u+v,u+2v\}$). 
    For example, for any (low-degree) polynomial map $\P \colon \F^n \to \F^m$ with $|\F| \ge 3$, the image of the ``parabola graph'' $(\P,\P^2) \colon \F^n \to \F^{2m}$ is a cap set. Indeed, 
    if $(w,w^2) = c(v,v^2) + (1-c)(u,u^2)$ for distinct $u,v,w \in \Image(\P)$ and $c \in \F$ (so $c \notin \{0,1\}$), where $x^2$ denotes coordinate-wise square, then equating the first $m$ coordinates gives $w=cv+(1-c)u$, and so equating the last $m$ coordinates implies, after manipulation, $c(1-c)(v_i-u_i)^2=0$ for every $i \in [m]$; thus, $v=u$, a contradiction.
    We refer to~\cite{EL23,Hir98,HT16} for more on the rich topic of constructions of cap sets in finite geometry.
\end{remark}

Henceforth, we denote
\[f(m,d) = (2m)^{2^{d(1+o(1))}}.\]

\begin{theorem}\label{thm:surj}
    Let $\P \in \Poly_d^m(\F)$ with $d < \ch(\F)$.
    If $\Image\P \sub \F^m$ does not contain a line then
    $\rk_{d/2}(\P) \le f(m,d)$.
\end{theorem}

Theorem~\ref{thm:surj} is a special case of a theorem about the degree of curves contained in the image of a polynomial map.
We next recall the definition of the \emph{univariate degree} of a polynomial map $\P$, 
which can be thought of as a combinatorial measure of ``complexity'' of the image of $\P$.
In particular, for a single polynomial $P$, it can be thought of as a measure of surjectivity---the larger it is, the further $P$ is from being surjective.

\begin{definition}[Univariate degree]\label{def:univariatedegree}
    For $\P \colon \F^n \to \F^m$, 
    \[\udeg(\P)=\min\{\deg(\U) \setmid \U \colon \F\to\F^m \text{ non-constant} \colon \Image\U \sub\Image\P\}.\]
    By convention, if $\P \colon \F^n \to \F^m$ is constant as a map then $\udeg(\P) = \infty$.\footnote{We require $U$ to be non-constant as a map, though note that a (minimum-degree) witness to $\udeg(\P)$, for $\P$ a non-constant map, must be a reduced polynomial, for which there is no distinction between being non-constant as a map or as a polynomial.} 
\end{definition}

We include a brief list of properties of univariate degree in Section~\ref{subsec:ud-properties} below.

\begin{remark}
    The special case of Definition~\ref{def:univariatedegree} where $(n,m,\udeg(P))=(1,1,1)$ is already very interesting. Indeed, note that for any finite field $\F$, a \emph{univariate} polynomial $P \colon \F\to\F$ satisfies $\udeg(P)=1$ if and only if $P$ induces a permutation of $\F$; such polynomials are naturally called \emph{permutation polynomials}, 
    and have attracted much attention 
    (see e.g.~\cite{Hou15} and the references within).
\end{remark}

Our main result in this section is as follows.

\begin{theorem}\label{thm:rk-bd}
    For every $\P \in \Poly_d^m(\F)$ with $d < \ch(\F)$
    and $u=\udeg(\P)$, we have 
    $\rk_{d/u}(\P) \le f(m,d)$.
\end{theorem}

We highlight two extreme special cases. 
One is where $\udeg(\P)$ is very large, namely $\udeg(\P)>\deg(\P)/2$,
in which case we directly obtain a small depth-$3$ (homogeneous) arithmetic formula for $\P$:
each component $P_i$ is expressed as a linear combination of (the same) few products of  degree-$1$ polynomials.
\begin{corollary}
    For $\P \in \Poly_d^m(\F)$ with $d < \ch(\F)$,
    if $\udeg(\P)>d/2$ then
    $\rk_{1}(\P) \le f(m,d)$.
%
\end{corollary}

\newcommand{\A}{\vec{A}}

The other extreme special case is Theorem~\ref{thm:surj}, which is the special case where $\udeg(\P)>1$.
Indeed, $\udeg(\P) = 1$ means $\Image\U \sub \Image \P$ for some univariate $\U \colon \F\to\F^m$ with $\deg(\U) = 1$, and observe that $\deg(\U)=1$ if and only if $\Image\U$ is a line.
Indeed, $\deg(\U)=1$ means $U_i(t) = ta_i+b_i$ with $a_i,b_i \in \F$, or equivalently, $\U(t) = ta+b$ with $a,b \in \F^m$, where $a \neq 0$.
Thus, if $\Image\P$ does not contain a line then $\udeg(\P) \ge 2$,
so Theorem~\ref{thm:rk-bd} implies $\rk_{d/2}(\P) \le f(m,d)$ (that is, Theorem~\ref{thm:surj}).

\subsection{Properties of univariate degree}\label{subsec:ud-properties}

Recall the definition of univariate degree for a polynomial map $\P \colon \F^n \to \F^m$:
\[\udeg(\P)=\min\{\deg(\U) \setmid \U \colon \F\to\F^m \text{ non-constant} \colon \Image\U \sub\Image\P\}.\]

\begin{example}\label{ex:udeg-power}
    The univariate $x\mapsto x^k \colon \F_p \to \F_p$ has $\udeg(x^k) = \gcd(k,p-1)$.
    To see this, note that $\Image(x^k)=\Image(x^g)$ with $g:=\gcd(k,p-1)$, so $\udeg(x^k) = \udeg(x^g) \le g$.
    To prove $\udeg(x^g) \ge g$, write $ag=p-1$, and let $U$ be a nonconstant univariate with $\Image(U) \sub \Image(x^g)$. If $\deg(U)<g$ then $A:=U^a$ satisfies $\deg(A) < ag = p-1$ and $\Image(A) \sub \Image(x^{p-1}) = \{0,1\}$; this is a contradiction ($0 = \sum_{x \in \F_p} A(x) = |\{x \in \F_p \setmid A(x) \neq 0\}| \pmod{p}$, so $A$, and $U$, are constants).
\end{example}

We first show that one cannot in general hope to bound $\rk_t(P)$ independently of $n$ for $t < \deg(P)/\udeg(P)$.
We have $\udeg(Q^k) \ge \udeg(x^k)$ for any (multivariate) polynomial $Q$, with equality for any surjective $Q$ (see Proposition~\ref{prop:udeg-prop} (\ref{item:udeg-prop-comp}));
moreover, 
$\udeg(x^k) = \gcd(k,p-1)$ assuming $\F=\F_p$ prime. 
Given a prime $p$, a divisor $k$ of $p-1$, and an integer $t \ge 2$,
let $Q \in \Poly_t(\F_p)$ be a random $n$-variate polynomial over $\F_p$ of degree $t$,
and set $P=Q^k$.
Since $Q$ is surjective with high probability, 
we have $\udeg(P) = \udeg(x^k) = k$ (recall Example~\ref{ex:udeg-power}), 
so $\deg(P)/\udeg(P) = t$.
On the other hand, with high probability and since $\deg(Q)=t$, we have that $\rk_{t-1}(P)$ is unbounded as $n \to \infty$ (see, e.g., Example~1.4 in~\cite{Karam23} or~\cite{BHL12,BGY20}).
Let us note that $\rk_t(P) \le 1$, since $P$ is a product of polynomials of degree at most $t$.
We also note that similar arguments lead to the same conclusion about $t=\deg(P)/\udeg(P)$ being the threshold for Green-Tao's notion $\rank_t(\P)$ as well.


Below is a brief list of observations regarding the notion of univariate degree, which we include for the sake of completeness.
For simplicity, we state them for the case of a single polynomial ($m=1$).
We will use the notation $\overline{P}$ for the unique \emph{reduced} polynomial equal to $P$ as a function, where reduced means that each variable has degree smaller than $|\F|$ (since $x^{|\F|} \equiv x$ as functions over $\F$).\footnote{That is, $\overline{P} = P \pmod{x_1^{|\F|}-x_1,\ldots,x_n^{|\F|}-x_n}$.}
Observe that $\udeg(P)=\udeg(\overline{P})$, since $\Image(P)=\Image(\overline{P})$.

\begin{proposition}\label{prop:udeg-prop}
For $P \in \Poly(\F)$, non-constant as a map,
the univariate degree $\udeg(P)$ satisfies the following properties: 
\begin{enumerate}
    \item $\udeg(P) \leq \min_i \deg_i(P)$ ($\le \deg(P)$),
    over all $i$ with $\deg_i(\overline{P})>0$.
    \item $\udeg(P) < |\F|$.  
    \item\label{item:udeg-prop-comp} $\udeg(F(\X)) \ge \udeg(F)$, for any $F \in \Poly$, with equality if $\vec{X}$ is surjective.
    %
    \item $\udeg(aP+b) = \udeg(P)$ for any $a,b \in \F$ with $a \neq 0$. 
    \item $\udeg(P^2) \ge 2$ if $\ch(\F) > 2$, with equality if $P$ is surjective.
\end{enumerate}
\end{proposition}
\begin{proof}
In order:
\begin{enumerate}
    \item It suffices to assume $P$ is reduced, as 
    $\udeg(P) = \udeg(\overline{P})$ and $\deg_i(\overline{P}) \le \deg_i(P)$.
    Let $i$ with $k:=\deg_i(P)>0$, and for concreteness assume $i=1$.
    Write $P$ as a univariate polynomial in $x_1$:
    $P=\sum_{j=0}^k x_1^j C_j(x_2,\ldots,x_n)$. 
    Since $C_k$ is both a nonzero polynomial and reduced (as $P$ is), it is nonzero as a function as well:
    $C_k(c) \neq 0$ for some $c \in \F^{n-1}$. Thus, the univariate polynomial $U(t)=P(t,c)=\sum_{j=0}^k t^j C_j(c)$ is of degree $k$. 
    Since $\Image U = \{P(t,c) \setmid t \in \F\} \sub \Image P$
    and $k>0$, we have $\udeg(P) \le k$.
    \item $\udeg(P)=\udeg(\overline{P}) \le \min_i \deg_i(P) < |\F|$ by the previous item.
    \item We may assume $F$ is non-constant. 
    For $U$ witnessing $\udeg(F(\X))$, we have $\Image(U) \sub \Image(F(\X)) \sub \Image(F)$, so $\udeg(F) \le \deg(U) = \udeg(F(\X))$.
    Moreover, if $\vec{X}$ is surjective then $\Image(F) = \Image(F(\X))$, and generally, $\udeg(P)$ depends only on $\Image(P)$.
    \item If a univariate $U$ satisfies $\Image(U) \sub \Image(P)$,
    then $\Image(aU+b) \sub \Image(aP+b)$, so $\udeg(aP+b) \le \deg(aU+b) = \deg(U)$,
    which shows $\udeg(aP+b) \le \udeg(P)$ if we choose $U$ that witnesses $\udeg(P)$.
    The equality follows by symmetry.
    \item $\udeg(P^2) \ge \udeg(x^2)$,
    and we claim that $\udeg(x^2) = 2$.
    It is not surjective as $\gcd(2,|\F|-1) = 2$, using that $|\F|$ is odd.
    Thus, $2 \le \udeg(x^2) \le \deg(x^2) = 2$.
    Moreover, by Item~\ref{item:udeg-prop-comp}, if $P$ is surjective then $\udeg(P^2) = \udeg(x^2)$.
\end{enumerate}
\end{proof}

\begin{remark}
    Univariate degree is related to the concepts of subspace-evasive sets and variety-evasive sets over finite fields, which are sets that have small intersections with, in particular, any line and any curve of low degree, respectively. (Explicit constructions were given by Dvir, Koll\'ar, and Lovett~\cite{DL12,DKL14}.)
    In particular, if $\Image(\P)$ evades curves in the sense that
    $|\Image(\P) \cap C| < |\F|/u$ for every polynomial curve $C \sub \F^m$ of degree at most $u$, 
    then $\udeg(\P) > u$. (Indeed, otherwise there is a curve $C \sub \Image(\P)$ with $\deg(C) \le u$, so $|C| \ge |\F|/u$.)
    However, the converse does not hold---large $\udeg(\P)$ is about $\Image(\P)$ avoiding containment of entire curves, and so is a weaker property than $\Image(\P)$ avoiding any significant intersections with curves. 
\end{remark}

\begin{example}\label{ex:uni-graph}
    Polynomial maps of high univariate degree are not hard to come by.
    For example, the graph of any univariate polynomial is a polynomial map of maximal univariate degree. That is, over any finite field $\F$ with $|\F|>d(d-1)$, if $F$ is any univariate polynomial of degree $d \ge 2$ over $\F$, then the polynomial map $x \mapsto (x,F(x)) \colon \F \to \F^2$ has univariate degree $d$. 
    (Recall the parabola graph example from Remark~\ref{remark:cap-sets}.)
    More generally in $n$ variables, for any surjective $Q \in \Poly_k(\F)$, the $n$-variate polynomial map $\P(x)=(Q(x),F(Q(x))) \in \Poly_{kd}^2(\F)$, which has the same image as the graph of $F$, satisfies $\udeg(\P)=d$.
    The upper bound $\udeg(\P) \le d$ is easy, as $(t,F(t))$ is a polynomial curve in $\Image(\P)$, so let us show the lower bound.
    If $\U \colon \F\to\F^2$ has $\deg(\U) \le d-1$ and $\Image(\U) \sub \Image(\P)$, 
    and we write $\U=(A,B)$ for univariates $A$ and $B$, which are (reduced) polynomials as $\deg(A),\deg(B) \le \deg(\U) < d \le |\F|$. 
    As functions, $B \equiv F \circ A$, and $A$ is non-constant as otherwise $B$, and thus $\U$, would be constant.
    Since $\deg(F \circ A) \le d\deg(A) < |\F|$, $F\circ A$ is a (reduced) polynomial as well,
    so $\deg(B)=\deg(F(A)) = d\deg(A) \ge d$, a contradiction.    
\end{example}

\begin{example}
    For graphs of multivariate polynomial maps $\vec{F}$, univariate degree is already a powerful notion;
    indeed, it suffices to encode the existence of solutions to systems of polynomial equations.
    Suppose $\vec{F}$ is a homogeneous polynomial map of degree $d \ge 2$, meaning all its components $F_i$ are degree-$d$ forms, over a finite field $\F$ with $|\F|>d(d-1)$.
    The graph of $\vec{F}$ is the polynomial map $\P(x) = (x,\vec{F}(x))$. 
    We show that the homogeneous system of equations $\vec{F}(x)=\vec{0}$ has a nonzero solution $x \neq 0$ if and only if $\udeg(\P)<d$.
    
    In one direction, if there is a vector $v \neq 0$ such that $\vec{F}(v)=\vec{0}$, 
    then the line $\{(tv,0) \setmid t \in \F\}$ is contained in $\Image(\P)$, meaning $\udeg(\P)=1$ ($<d$).
    In the other direction, if a non-constant polynomial curve $t\mapsto (\A(t),\vec{F}(\A(t)))$ has degree smaller than $d$
    then, similarly to Example~\ref{ex:uni-graph}, $\vec{F}(\A)$ and $\A$ are (reduced) polynomial maps.
    If $\A(t)$ has top-degree homogeneous part $ut^e$, for a vector $u \neq 0$ and degree $1 \le e<d$, then $\vec{F}(\A(t))$ has top-degree homogeneous part $\vec{F}(u)t^{ed}$. Since $\vec{F}(u) \neq \vec{0}$ by assumption, $\deg(\vec{F}(\A)) = ed \ge d$, a contradiction.
    
    Let us finally note that in this example the univariate degree takes only two values, $\udeg(\P) \in \{1,d\}$. However, if $\vec{F}$ is not homogeneous then, depending on the existence of nonzero solutions to the homogeneous equations given by various homogeneous parts of $\vec{F}$, $\udeg(\P)$ can take all other values in $\{1,2,\ldots,d\}$ (e.g., the graph of $\vec{F}(x,y)=(x^r,y^d)$, with $r \le d$, has univariate degree $r$).
\end{example}

\subsection{Polynomial decompositions}

First, we will need the following immediate property of weak $\e$-regular tuples.
Note that for an arbitrary map $\X \colon \F^n \to \F^r$,
if a line $\ell \sub \F^r$ in a given direction is chosen uniformly at random,
then $\Exp_\ell \Pr[\X \in \ell] = q^{-(r-1)}$.
We show that if $\X$ is weak regular, then lines that are
somewhat popular (that is, close to being as popular as a random line)
must be fully contained in the image of $\X$.

\begin{lemma}\label{lemma:contain-line}
    Let $\X \sub \Forms^r(\F_q)$ be weak $\e$-regular.
    For every line $\ell \sub \F^r$ in direction $e_1$,
    if $\Pr[\X \in \ell] > \e q^{-(r-1)}$
    then $\ell \sub \Image\X$.
\end{lemma}
\begin{proof}
    By the definition of weak $\e$-regularity (recall~(\ref{eq:line})),
    for every $y \in \ell$ we have
    \[\Pr[\X=y] \ge \Pr[\X \in \ell]/q - \e q^{-r} > 0.\] 
    Thus $y \in \Image\X$,
    which completes the proof.
\end{proof}

Motivated by Lemma~\ref{lemma:contain-line},
we proceed to prove results about popular lines in the image of polynomial decompositions (equivalently, of polynomials in not necessarily independent variables).
These results do not assume weak regularity.

We say that a polynomial decomposition $P = F(\X)$ is
\emph{irredundant} if there is no polynomial $F'$ with $\Supp F' \subsetneq \Supp F$ such that $F'(\X)=F(\X)$.\footnote{As usual, this is an equality of polynomials. Recall that $\Supp F$ denotes the support of $F$, meaning the set of monomials appearing with a nonzero coefficient in $F$.}
Note that the definition of a weak $\e$-regular decomposition (Definition~\ref{def:reg}) is independent of the choice of $F$, so there is no loss of generality in assuming it is irredundant, as we will later do.

First, we show that a decomposition of a polynomial $P$ into forms need not involve terms of higher degree than that of $P$.

\begin{lemma}[irredundant degree]\label{lemma:semi-hom}
    For $P \in \Poly$,
    let $P = F(\X)$, with $\X \sub \Forms$, be an irredundant decomposition.
    For every monomial $M$ in the support of $F$,
    $\deg M(\X) \le \deg P$.
%
\end{lemma}
\begin{proof}
    Since every $X_i$ is homogeneous, for every monomial $M \in \Supp(F)$, all the monomials that result from expanding $M(\X)$ are of the same degree.
    Thus, for each $k$, the degree-$k$ part of $F(\X)$ is equal to 
    $\sum M(\X)$ with the sum over all monomials $M \in \Supp(F)$ with $\deg M(\X)=k$.
    Therefore, if $F'$ is obtained from $F$ by removing from its support all monomials $M$ with $\deg M(\X)=k$, for $k>\deg F(\X)$, then we still have $F'(\X)=F(\X)$.
    We deduce that if there exists a monomial $M \in \Supp(F)$ with $\deg M(\X) > \deg F(\X)$, then $P=F(\X)$ is not an irredundant decomposition, completing the proof.
\end{proof}

\begin{corollary}\label{coro:deg-X}
    For any decomposition $P \sub \F[\X]$,
    with $P \in \Poly$ and $\X \sub \Forms$,
    if $P \nsubseteq \F[\X']$ 
    then $\deg(X_1) \le \deg(P)$.
\end{corollary}
\begin{proof}
    Without loss of generality, $P = F(\X)$ is an irredundant decomposition.
    We have $\deg_1(F) \ge 1$, since $P \nsubseteq \F[\X']$.
    Apply Lemma~\ref{lemma:semi-hom} on a monomial $M$ in the support of $F$ with $\deg_1(M) \ge 1$. Then $\deg(X_1) \le \deg M(\X) \le \deg(P)$. 
\end{proof}

We next prove 
that in the image of any polynomial map $\X$ one can find a popular line that avoids a given low-degree zero set.

\begin{lemma}[Schwartz-Zippel for decompositions]\label{lemma:S-Z}
    Let $P \in \Poly_d(\F_q)$, with $d<q$, be nonzero.
    If $P = F(\X)$ with $\X \in \Poly^k$ (and $F \in \Poly$),
    then there exists $y \in \F_q^k$ such that $F(y) \neq 0$ and 
    \[ \Pr[F\neq 0] \cdot \Pr[\X=y] \ge (1-d/q)q^{-k}.\]
\end{lemma}
In particular, $\Pr[\X=y] \ge (1-d/q)q^{-k}$.
Note that if $\X$ is perfectly equidistributed (e.g., independent variables), 
 we retrieve the usual Schwartz-Zippel lemma: $\Pr[F \neq 0] \ge 1-d/q$.
    \begin{proof}
        Consider $\{y \in \F^k \setmid F(y) \neq 0\}$, the complement of the zero set $\Z(F)$ of $F$.
        We have
        \[\Pr_{x}[P(x) \neq 0]
        = \Pr_{x}[\X \notin \Z(F)]
        = \sum_{y \notin \Z(F)}\Pr[\X = y]
        \le q^k\Pr[F \neq 0]\max_{y \notin \Z(F)}\Pr[\X = y].\]
        On the other hand, as $P$ is nonzero, the Schwartz-Zippel lemma gives
        $\Pr[P \neq 0] \ge 1-d/q$.
        Combining the above, we get
        $\Pr[F \neq 0]\max_{y \notin \Z(F)}\Pr[\X = y] \ge (1-d/q)q^{-k}$,
        as desired.
\end{proof}

We combine the above lemmas to find a popular line that the decomposition nontrivially depends on.
Denote by $\deg_1(F)$ the degree of $F \in \F[x_1,\ldots,x_r]$ in its first variable (that is, the largest exponent of $x_1$ in the monomials in the support of $F$).

\begin{corollary}\label{coro:nonconstant-line}
    Let $P \in \Poly_d(\F_q)$ with $d<q$. 
    If $P = F(\X)$ with $\X \in \Forms^r$
    is an irredundant decomposition with $\deg_1 F>0$,
    then there exists a line $\ell \sub \F_q^r$ in direction $e_1$ such that
    $F|_{\ell}$ is non-constant 
    and $\Pr[\X \in \ell] > (1-d/q)q^{-(r-1)}$.
\end{corollary}
\begin{proof}
    Put $\F=\F_q$.
    Write $F(t,\y) = \sum_{i \ge 0} t^i C_i(\y)$ with $C_i \in \Poly$, the unique representation of $F$ as a linear combination of the monomials in its support.
    We claim that the following holds for some $j \ge 1$:
    \begin{itemize}
        \item $C_j(\X')$ is a nonzero polynomial; indeed, otherwise we have $P = F(\X) = F(X_1,\X') 
        = C_0(\X')$; moreover, $\Supp C_0 \sub \Supp F$, and $\Supp C_0 \neq \Supp F$ by the assumption $\deg_1(F) > 0$. This contradicts the assumption that $P=F(\X)$ is an irredundant decomposition.
        \item $C_j(\X')$ is of degree smaller than $d$; indeed, since $\X \sub \Forms$ and the decomposition $P=F(\X)$ is irredundant,
        we have that $\deg(C_j(\X')) < \deg(X_1^j C_j(\X')) \le \deg(P) \le d$, using Lemma~\ref{lemma:semi-hom} in the second inequality.
    \end{itemize}
    Apply Lemma~\ref{lemma:S-Z} on the nonzero $C_j(\X') \in \Poly_{d-1}$ to obtain $y_0 \in \F_q^{r-1}$ such that $C_j(y_0) \neq 0$ and $\Pr[\X' = y_0] > (1-d/q)q^{-(r-1)}$.
    Observe that $\deg(F(t,y_0)) \ge j$ ($\ge 1$), so in particular, $F(t,y_0)$ is a non-constant univariate polynomial.
    Thus, the restriction $F|_{\ell}$ of $F$ to the line $\ell = \{(t,y_0) \setmid t\in\F\} \sub \F^r$ is non-constant.
    Moreover, observe that $\Pr[\X \in \ell] = \Pr[\X' = y_0]$. This completes the proof.
\end{proof}

Finally, we use the following easy (and tight) binomial coefficient bound.
\begin{claim}\label{claim:binomial}
    $\binom{n+t}{t} \le \frac32 n^t$ for every $n\ge 2$, $t\ge 1$.
\end{claim}
\begin{proof}
    By induction on $t$, with base case $\binom{n+1}{1}=n+1 \le \frac32 n$ as $n\ge 2$.
    For the induction step, $\binom{n+t+1}{t+1} = \binom{n+t}{t} \frac{n+t+1}{t+1} \le \frac32 n^t \cdot n$,
    using the bound $\frac{n+t+1}{t+1} \le n$, or equivalently, $1 \le t(n-1)$.
    
\end{proof}

\subsection{Proofs of the applications}


\begin{proof}[Proof of Theorem~\ref{thm:rk-bd}]
    Put $\F=\F_q$.
    Assume $\P$ is non-constant, as otherwise $\rk_t(\P) = 0$ by convention and we are done.
    Put $\e=1-d/q$.
    By Theorem~\ref{thm:WRL}
    (with $c=2$, so $q^{-(c-1)k} \le 1/q \le \e$), 
    $\P$ has a weak $\e$-regular decomposition $\P = \vec{F}(\X)$, with $\X \in \Forms^r$,
    of size $r \le (2m)^{2^{d(1+o(1))}}$.
    We will prove that $\deg(\X) \le \deg(\P)/\udeg(\P)$.
    
    We assume without loss of generality that each decomposition $P_i=F_i(\X)$ is irredundant (by removing monomials from the support of $F_i$ if necessary).
    Since $\P \sub \F[\X]$ but $\P \nsubseteq \F[\X']$, 
    there is a component $P_j$ of $\P$ such that $P_j \in \F[\X]\sm\F[\X']$,
    and therefore $\deg_1(F_j)>0$.
    Apply Corollary~\ref{coro:nonconstant-line} on $P_j \in \Poly_d(\F)$ to obtain a line $\ell \sub \F^r$ in direction $e_1$ such that $F_j|_\ell$ is non-constant and $\Pr[\X \in \ell] > (1-d/q)q^{-(r-1)}$.
    Note that $F_j|_\ell$ is non-constant as a function as well, since it is a reduced polynomial, that is, $\deg(F_j|_{\ell}) < |\F|$; indeed, 
    $\deg(F_j|_{\ell}) \le \deg_1(F_j) \le \deg(P_j)/\deg(X_1)\le d < \ch(\F) \le |\F|$.
    Apply Lemma~\ref{lemma:contain-line} on $\ell$, using the weak $\e$-regularity of $\X$ and our choice of $\e$, to deduce that 
    \begin{equation}\label{eq:line-prop}
      \ell \sub \Image\X \text{ and } \vec{F}|_{\ell} \text{ is non-constant.}
    \end{equation}
    Let $\U \colon \F \to \F^m$ be the curve 
    $\U = \vec{F}(\vec{L})$, where $\vec{L} \colon \F\to\F^r$ satisfies $\Image(\vec{L})=\ell$.
    (Explicitly, if $\ell = \{te_1 + y \setmid t \in \F\}$, 
    then $\vec{L}(t) = te_1 + y$,  so  $\U(t)=\vec{F}(\vec{L}(t))=\vec{F}(t+y_1,y_2,\ldots,y_r)$.)
    Then $\Image\vec{U}=\vec{F}(\Image(\vec{L}))=\vec{F}(\ell)$,
    which implies that
    \[\Image\vec{U} \sub \Image\vec{P} \text{ and } \vec{U} \text{ is non-constant},\]
    since $\vec{F}(\ell) \sub \vec{F}(\Image\X) = \Image\P$ 
    by~(\ref{eq:line-prop}).
    Thus, by definition, $\deg(\U) \ge \udeg(\P)$.

    Let $i \in [m]$; abbreviate $P=P_i$, $F=F_i$, and $U=U_i$, 
    so that $P=F(\X)$ and $U=F(\vec{L})$.
    We claim that the following chain of inequalities holds: 
    \[\deg(U) \le \deg_1(F) \le \deg(P)/\deg(X_1) = \deg(P)/\deg(\X).\]
    The first inequality is immediate from the construction of $U$. 
    The second inequality follows using Lemma~\ref{lemma:semi-hom}: as $P=F(\X)$, for every monomial $M=\prod_{j=1}^r x_j^{m_j(M)}$ in the support of $F$, we have
    $m_1(M)\deg(X_1) \le
    \deg(M(\X))
    \le \deg(P)$;
    since $\deg_1(F) = \max_{M} m_1(M)$ over $M \in \Supp F$, we indeed obtain that $\deg_1(F)\deg(X_1) \le \deg(P)$.
    The final equality follows from the construction of $\X$.
    We obtain the following degree bound for $\X$, using $i$ satisfying $\deg(U_i)=\deg(\U)$:
    \[\deg(\X) \le \deg(P)/\deg(U) \le \deg(\P)/\deg(\U) \le \deg(\P)/\udeg(\P).\]

    Finally, since for every component $P_i$ of $\P$
    we have an irredundant decomposition $P_i = F_i(\X)$,
    we have from Lemma~\ref{lemma:semi-hom}
    that $\deg M(\X) \le \deg(P_i) \le d$ for every $M \in \Supp(F_i)$.
    Let $\mathcal{R}=\{M(\X) \setmid M \in \bigcup_i \Supp(F_i)\}$.
    Observe that $\mathcal{R} \sub \Poly_d$ is a set of $d/u$-reducible polynomials such that $\P \sub \Span\mathcal{R}$, 
    so by Definition~\ref{def:t-rank} we have $\rk_{d/u}(\P) \le |\mathcal{R}|$.
    Since each member of $\mathcal{R}$ is a product of at most $d$ members of $\X$, we are done (using Claim~\ref{claim:binomial}, assuming $r\ge 2$ without loss of generality):
    \[\rk_{d/u}(\P) \le \binom{r+d}{d} \le \frac32 r^d
    \le (2m)^{2^{d(1+o(1))}}. \qedhere\]
\end{proof}

We are now ready to deduce Theorems~\ref{thm:formula-ub} and~\ref{thm:top-fan-in}.

\begin{proof}[Proof of Theorem~\ref{thm:formula-ub}]
    Put $t=\lfloor d/u \rfloor$ ($\ge 1$). 
    Let $P$ be an $n$-variate polynomial with $\deg(P) \le d$ and $\udeg(P)= u$.
    Assume $P$ is non-constant, as otherwise there is nothing to prove. 
    Then $P = R_1+\cdots+R_{\rk_t(P)}$ 
    where each $R_i \in \Poly_d$ is a product of polynomials of degree at most $t$, that is, $R_i = \prod_{j \le d_i} F_{i,j}$ with $\deg(F_{i,j}) \le t$ and $d_i \le d$, by Definition~\ref{def:t-rank}.
    Since, trivially, $L(F_{i,j}) \le O(t\binom{n+t}{t})
    \le O(d n^t)$ (Claim~\ref{claim:binomial}, assuming $n\ge 2$),
    we deduce that 
    \[L(P) \le O\big(\sum_{i,j} L(F_{i,j})\big) \le O(d^2 \cdot \rk_t(P) \cdot n^t) 
    \le n^{t+o(1)},\]
    where the last inequality uses Theorem~\ref{thm:rk-bd}, 
    so $\rk_t(P) = \rk_{d/u}(P) \le f(1,d)$,
    and the statement's assumption $d \le \frac12\log_2\log_2(n)$, to bound 
    \[d^2\rk_t(P) \le 2^{2^{d(1+o(1))}} 
    \le 2^{(\log_2(n))^{\frac12+o(1)}} 
    = n^{(\log_2(n))^{-(\frac12-o(1))}} 
    \le n^{o(1)}.\]
    
    Moreover, if $\P \colon \F^n \to \F^m$ is a (non-constant) polynomial map with $\deg(\P) \le d$ and $\udeg(\P) = u$ then, by the definition of $r=\rk_t(\P)$, each component $P_i$ is a linear combination of $t$-reducible polynomials $R_1,\ldots,R_r \in \Poly_d$. 
    Similarly to the argument above, $S(R_i) \le O(d\binom{n+t}{t}) \le O(dn^t)$. 
    We deduce that
    \[S(\P) 
    \le \sum_{i=1}^r S(R_i) + O(mr) 
    \le O(drn^t) + O(mr) 
    \le O(dm r \cdot n^t)
    \le n^{t+o(1)}\]
    where the last inequality uses Theorem~\ref{thm:rk-bd}, 
    so $r \le \rk_{d/u}(\P) \le f(m,d)$,
    and the statement's additional assumption $m \le 2^{\log^{o(1)}(n)}$,
    to bound 
    \[dm r \le (2m)^{2^{d(1+o(1))}} 
    = (2^{2^{d(1+o(1))}})^{\log_2(2m)} 
    \le n^{(\log_2(n))^{-(\frac12+o(1))}} 
    \le n^{o(1)}. \qedhere\]
\end{proof}

\begin{proof}[Proof of Theorem~\ref{thm:top-fan-in}]
    Let $P \sub \F[\P]$ with $\udeg(\P) = u$.
    As in the proof of Theorem~\ref{thm:rk-bd}, using the weak regularity lemma 
    we obtain a decomposition $\P \sub \F[\X]$ with $\X \in \Forms^s$ of size $s \le (2m)^{2^{d(1+o(1))}}$ such that $\deg(\X) \le d/u$.
    Observe that $P \in \F[\P] \sub \F[\X]$.
    Using Lemma~\ref{lemma:semi-hom} 
    with an irredundant representation (by removing monomials if necessary), 
    it follows that we can write 
    \begin{equation}\label{eq:d4-formula}
        P = \sum_{i=1}^r c_i\prod_j Q_{i,j}
    \end{equation}
    with $Q_{i,j} \in \X$, $c_i \in \F$, $\deg(\prod_j Q_{i,j}) \le d$, and $r \le \sum_{i=0}^d s^i \le (2m)^{2^{d(1+o(1))}}$.
    Note that $\deg(Q_{i,j}) \le \deg(\X) \le d/u$.
    
    We claim that we can guarantee that the number of multiplicands in each product $\prod_j Q_{i,j}$ is at most $2u$. 
    Note that we can express each product as a product of homogeneous polynomials $\prod_{j=1}^{a_i} Q'_{i,j}$ all of whose degrees, except at most one of them, satisfy $\frac12(d/u) < \deg(Q'_{i,j}) \le d/u$ (note that $Q'_{i,j}$ are not necessarily members of $\X$); indeed, this follows by repeatedly using the fact that if $\deg(Q_{i,j_1}), \deg(Q_{i,j_2}) \le \frac12(d/u)$ then $\deg(Q_{i,j_1}Q_{i,j_2}) \le d/u$.
    So, $d \ge \sum_{j=1}^{a_i} \deg(Q'_{i,j}) > (a_i-1)\frac12(d/u)$, hence $a_i \le 2u$ as claimed.
    Thus, (\ref{eq:d4-formula}) gives a $\sum^{[r]}\prod^{[2u]}\sum\prod^{[d/u]}$-formula for $P$, as needed.
    
    Finally, if $P$ is homogeneous then, since every member of $\X$ is homogeneous, without loss of generality the products in the formula satisfy $\deg(\prod_j Q_{i,j})=\deg(P)$ for every $i$, and therefore the formula is homogeneous, as desired.
\end{proof}

\section{Open Problems}\label{sec:Open}

\begin{problem}
    It remains an open problem to improve the bound for the weak regularity lemma in Theorem~\ref{thm:WRL-into} (in terms of $\e$, $m$, and $d$), or its stronger version in Theorem~\ref{thm:WRL} (in terms of $c$, $m$, and $d$).
    Any improvement here would directly translate to improvements in all applications in Section~\ref{sec:Apps}.
    Conversely, it is also open whether there is a matching lower bound on the number of parts, 
    perhaps analogously to the famous matching lower bound of Gowers~\cite{Gow97} for the graph regularity lemma (see also~\cite{MS16}), or that for the weak regularity lemma (see~\cite{CF12}).
    Our proof has two main quantitative bottlenecks: the rank-regularity bound in Theorem~\ref{theo:rankular} and the rank-to-bias estimate in Theorem~\ref{theo:MZ}. It would be interesting to determine the optimal bounds for each, as well as to obtain direct lower bounds for weak regular decompositions.
\end{problem}

\begin{problem}
We leave it as an intriguing open problem to obtain improved upper bounds on $\rk_{d/2}(\P)$
for polynomial maps $\P \in \Poly_d^m(\F)$ with a line-free image. 
Conversely, it would also be interesting to obtain superlinear $\omega(m)$ lower bounds on $\sup_\P \rk_{d/2}(\P)$ for such degree-$d$ $\P$,
seeing as the parabola-graph example $(x_1,\ldots,x_r,x_1^2,\ldots,x_r^2)$ from Remark~\ref{remark:cap-sets} has linearly independent components, giving $\rk_{d/2}(\P) \ge 2r = m$ for any $d \ge 2$.
In fact, even the case of quadratic maps ($d=2$) seems nontrivial; 
using the same proof method together with the exact rank–bias relation for quadratic forms gives
the bound $\rk_1(\P) \le O(m^2)$ for line-free quadratic maps, and it would be interesting to determine the truth here.

These bounds should be independent of $n$.
For comparison, a monomial grouping argument gives the tight bound\footnote{The upper bound follows by grouping monomials according to a divisor of degree $\lceil d/2 \rceil$, and the matching lower bound follows by counting parameters.} $\rk_{d/2}(\P) \le O_d(mn^{\lceil d/2\rceil})$, for any $\P \in \Poly_d^m(\F)$ with $d\ge 2$ and $m \le O_d(n^{\lfloor d/2 \rfloor})$, so it is the assumption of a line-free image that allows for a bound independent of $n$.

These problems have a somewhat similar nature to quantitative versions of the celebrated Stillman's conjecture in commutative algebra~\cite{AH20B,CLM25}, but with a more combinatorial flavor.
A speculative, and likely difficult, conjecture here would be that, over any finite field $\F$ with $2 \le d<\ch(\F)$, any polynomial map $\P \in \Poly_d^m(\F)$ with a line-free image satisfies
$\rk_{d/2}(\P) \le (2m)^{O(d)}$. 
\end{problem}

\paragraph{Acknowledgments:} We are grateful to the anonymous referees for multiple useful comments.

\bibliographystyle{alpha}

\end{document}